\documentclass[reqno]{amsart}
\usepackage{amsmath,amsthm,amscd,amssymb,amsfonts, amsbsy}
\usepackage{latexsym, color, enumerate}
\usepackage{pxfonts}

\theoremstyle{plain}
\newtheorem{theorem}{Theorem}[section]
\newtheorem{lemma}[theorem]{Lemma}
\newtheorem{corollary}[theorem]{Corollary}
\newtheorem{proposition}[theorem]{Proposition}

\theoremstyle{definition}
\newtheorem{definition}[theorem]{Definition}

\newtheorem{example}[theorem]{Example}
\newtheorem{assumption}[theorem]{Assumption}

\theoremstyle{remark}
\newtheorem{remark}[theorem]{Remark}

\newcommand{\supp}{\operatorname{supp}}
\newcommand{\dist}{\operatorname{dist}}
\newcommand{\diam}{\operatorname{diam}}

\numberwithin{equation}{section}

\newcommand{\bI}{\mathbb{I}}

\newcommand{\bR}{\mathbb{R}}

\newcommand\cA{\mathcal{A}}
\newcommand\cB{\mathcal{B}}

\newcommand\cD{\mathcal{D}}
\newcommand\cE{\mathcal{E}}

\newcommand\cN{\mathcal{N}}

\newcommand\cM{\mathcal{M}}

\providecommand{\set}[1]{\{#1\}}

\providecommand{\abs}[1]{\lvert#1\rvert}
\providecommand{\Abs}[1]{\left\lvert#1\right\rvert}
\providecommand{\bigabs}[1]{\bigl\lvert#1\bigr\rvert}

\providecommand{\norm}[1]{\lVert#1\rVert}
\providecommand{\Norm}[1]{\left\lVert#1\right\rVert}
\providecommand{\bignorm}[1]{\bigl\lVert#1\bigr\rVert}

\renewcommand{\vec}[1]{\boldsymbol{#1}}

\def\Xint#1{\mathchoice
	{\XXint\displaystyle\textstyle{#1}}%
	{\XXint\textstyle\scriptstyle{#1}}%
	{\XXint\scriptstyle\scriptscriptstyle{#1}}%
	{\XXint\scriptscriptstyle\scriptscriptstyle{#1}}%
	\!\int}
\def\XXint#1#2#3{{\setbox0=\hbox{$#1{#2#3}{\int}$}
		\vcenter{\hbox{$#2#3$}}\kern-.5\wd0}}
\def\dashint{\Xint-}%For the average integral symbol

\newcommand{\p}{\partial}
\newcommand{\epsi}{\varepsilon}

\begin{document}



	\title[mixed boundary value problem]{optimal regularity for a Dirichlet-conormal problem in Reifenberg flat domain}
	
	\author[Jongkeun Choi]{Jongkeun Choi}
	
	\address{School of Mathematics\\Korea Institute for Advanced Study\\
		85 Hoegiro\\Dongdaemun-gu\\Seoul 02455\\Republic of Korea
		}
	
	\email{jkchoi@kias.re.kr}
	
	\author[Hongjie Dong]{Hongjie Dong}	
	
	\address{
Division of Applied Mathematics, Brown University, 182 George Street, Providence, RI 02912, USA}
	
	\email{hongjie\_dong@brown.edu}
\thanks{H. Dong and Z. Li were partially supported by the NSF under agreement DMS-1600593.}
	
	\author[Zongyuan Li]{Zongyuan Li}
	
	\address{Division of Applied Mathematics, Brown University, 182 George Street, Providence, RI 02912, USA}
	
	\email{zongyuan\_li@brown.edu}
	
\begin{abstract}
We study the divergence form second-order elliptic equations with mixed Dirichlet-conormal boundary conditions. The unique $W^{1,p}$ solvability is obtained with $p$ being in the optimal range $(4/3,4)$. The leading coefficients are assumed to have small mean oscillations and the boundary of domain is Reifenberg flat. We also assume that the two boundary conditions are separated by some Reifenberg flat set of co-dimension $2$ on the boundary.
\end{abstract}
%\today
\maketitle

%\tableofcontents

%========================================
\section{Introduction}
%========================================
In this paper, we discuss the mixed boundary value problem for second-order elliptic operators:
\begin{equation}		\label{180126@eq1}
\begin{cases}
Lu =f+ D_i f_i  & \text{in }\, \Omega,\\
Bu = f_i  n_i  & \text{on }\, \cN,\\
u =0 & \text{on }\, \cD,
\end{cases}
\end{equation}
where $\Omega$ is a domain (not necessarily bounded) in $\bR^d$, $d\ge 2$ with the boundary divided into two non-intersecting portions $\cD$ and $\cN$.
The differential operator $L$ is in divergence form acting on real valued functions $u$ as follows:
\begin{equation*}
Lu=D_i(a_{ij}(x)D_j u+b_i(x) u)+\hat{b}_i(x) D_i u+c(x)u.
\end{equation*}
Here, all the coefficients are assumed to be bounded measurable, and the leading coefficients $a_{ij}$ are symmetric and uniformly elliptic.
We denote by $Bu=(a_{ij}D_j u+b_i u)n_i$ the conormal derivative of $u$ on $\cN$ associated with the operator $L$.
Dirichlet and conormal boundary conditions are prescribed on the portions $\cD$ and $\cN$ respectively, which are separated by their relative boundary $\Gamma\subset\p\Omega$. Both the equation and the boundary conditions are understood in the weak sense. For precise definition, see Definition \ref{def-weak-sln}.

As is well known, solutions to purely Dirichlet/conormal boundary value problems  are smooth when coefficients, data, and boundaries of domains are smooth.
However, for mixed boundary value problems, such a regularity result does not hold near the interface $\Gamma$, and the regularity of solutions depends also on that of $\Gamma$ and the way two boundary conditions meet
(e.g., the meeting angle and certain compatibility conditions).
For instance, the best possible regularity of derivatives of solutions to \eqref{180126@eq1} is
$$
Du \in L_p \quad \text{for }\, p<4
$$
when the two boundary portions meet tangentially (the angle between $\cD$ and $\cN$ is $\pi$); see Example \ref{eg-classical} for a classical counterexample.
In this paper, we investigate minimal regularity assumptions of $a_{ij}$, $\partial \Omega$, and $\Gamma$, which guarantee the above optimal regularity as well as the solvability of the mixed problem \eqref{180126@eq1}.

Regularity theory for mixed problems has been studied for a long time.
For the case when the two boundary portions $\cD$ and $\cN$ meet tangentially, we refer the reader to Shamir \cite{Sh} and Savar\'e \cite{Sav}.
In \cite{Sh}, the author proved $W^{1,4-\varepsilon}$ regularity for non-divergence form elliptic equations with smooth coefficients in half space.
He also obtained $W^{s,p}$ regularity on a smooth bounded domain with the indices $p>4$ and $s<1/2+2/p$.
At one end, the optimal $C^{1/2-\epsi}$-H\"older regularity can be obtained by passing $p\nearrow\infty$, which improved a general H\"older regularity result of De Giorgi's type by Stamppachia in \cite{Sta}.
It is also worth mentioning that in \cite{Sav}, the author proved optimal regularity in Besov space $B^{3/2}_{2,\infty}$ for the divergence form elliptic equations with Lipschitz coefficients on a $C^{1,1}$ domain.

For the case when $\cD$ and $\cN$ do not meet tangentially, we refer the reader to I. Mitrea-M. Mitrea \cite{MR2309180}, where the authors studied the mixed problem  \eqref{180126@eq1} with
$L u=\Delta u$.
They proved the $W^{1,p}$ solvability with
\begin{equation}		\label{190201@eq1}
\frac{3}{2+\varepsilon}<p<\frac{3}{1-\varepsilon} \quad \text{for some }\, \varepsilon=\varepsilon(\Omega, \cD, \cN)\in (0,1)
\end{equation}
on the so-called creased domains in $\bR^d$, $d\ge 3$, which means that $\cD$ and $\cN$ are separated by a Lipschitz interface and the angle between $\cD$ and $\cN$ is less than $\pi$.
This class of domains was introduced by Brown in \cite{B} to answer a question raised by Kenig in \cite{K} regarding the non-tangential maximal function estimate
$$
\norm{(\nabla u)^{*}}_{L_2(\p\Omega)}\le C\big(\norm{\nabla_{\rm{tan}} u}_{L_2(\cD)}+\norm{\p u/ \p n}_{L_2(\cN)}\big)
$$
of harmonic functions.
As mentioned in \cite{K}, the above regularity result can be false when $\Omega$ is smooth so that $\cD$ and $\cN$ meet tangentially, whereas it holds for purely Dirichlet/Neumann problem.
For further work in this direction, see \cite{MR2503013, OB} and the references therein.

In this paper, we work on the so-called ``Reifenberg flat'' domain, which is, roughly speaking, at every small scale the boundary is close to certain hyperplane. A Reifenberg flat domain is much more general than a Lipschitz domain with small Lipschitz constant: locally it is not given by a graph, and typically it contains fractal structures. The Reifenberg flat domain was introduced by Reifenberg in \cite{R} when he worked on the Plateau problem. Since then, there has been a lot of work on Reifenberg flat domains regarding minimal surfaces, harmonic measures, regularity of free boundaries, and divergence form elliptic/parabolic equations. An important fact for studying divergence form equation in such domains is that any small Reifenberg flat domain is a $W^{1,p}$-extension domain for every $p\in[1,\infty]$. Hence we have all the Sobolev inequalities up to the first order. For this result and the history of studying Reifenberg flat domains, one may refer to \cite{AMS}.

Notice that although on Reifenberg flat domain, neither the outer normal nor the trace operator of $W^{1,p}$ is defined, the weak formulation in Definition \ref{def-weak-sln} still makes sense due to the fact that no boundary integral term appears when $\Omega$ is smooth enough so that the outer normal and the trace operator are well defined.

We prove the solvability in Sobolev spaces $W^{1,p}$ and the $L_p$-estimates with $p$ being in the optimal range
$4/3<p<4$ for the mixed problem \eqref{180126@eq1} with BMO coefficients on Reifenberg flat domains.
The two boundary portions $\cD$ and $\cN$ are assumed to meet {\em{almost}} tangentially, which means $\cD$ and $\cN$ are separated by some Reifenberg flat set of co-dimension $2$ on the boundary.
We note that our result holds for both bounded and unbounded domains.
For the bounded domain case, we can further relax the assumptions on the source term.
As mentioned before, since Lipschitz domains with small Lipschitz constant are Reifenberg flat, our results can be applied also on creased domains.
Therefore, we see that in the restriction \eqref{190201@eq1}, the best possible range of $\varepsilon$ is $0<\varepsilon<1/4$ for creased domains with small  Lipschitz constant.

This paper is a continuation of \cite{DK11,DK12}, in which elliptic systems on Reifenberg flat domains with rough coefficients and purely Dirichlet/conormal boundary conditions were studied. See also the series \cite{BW1,BW2} regarding second-order equations on bounded domains. Our proof is mainly based on a perturbation argument suggested in \cite{Ca} by Caffarelli and Peral, by studying the level sets of maximal functions. The key step in our proof is to carefully design an approximation function near $\Gamma$, which combines the cut-off and reflection techniques in \cite{DK11, DK12}. Compared to the purely Dirichlet or purely conormal problems, the approximation function in our problem is less regular, which is only $W^{1,4-\epsi}$, not Lipschitz. This situation is similar to \cite{arXiv:1806.02635}, where dedicated decay rates of the level sets are required.

The paper is organized as follows. In Section 2, we introduce the basic notation, definitions, and assumptions. Our main results are given in Theorem \ref{thm-well-posedness} for both bounded and unbounded domains and in Theorem \ref{thm-bounded-domain} for bounded domains. In Section 3, we prove two useful tools for our problem: the local Sobolev-Poincar\'e inequality and the reverse H\"older inequality. Then in Section 4, we study a model problem, which is the $W^{1,4-\epsi}$ regularity of harmonic functions on the upper half space with mixed boundary conditions. With all these preparation, the proof of the main theorem including the approximation via cut-off and reflection, and the level set argument is presented in Section 5. In Section 6, we relax the regularity assumptions on the source term for the bounded domain case, mainly by solving a divergence form equation.
%========================================
\section{Notation and Main Results}
%========================================
Let $d$ be the space dimension. We write a typical point $x\in \bR^d$ as $x=(x',x'')$, where
$$x'=(x_1,x_2)\in \bR^2, \quad x''=(x_3,\cdots,x_d)\in \bR^{d-2}.$$
In the same spirit, for a domain $\Omega\subset \bR^d$ and $p,q\ge 1$, we define the anisotropic space $L_{p,x'}L_{q,x''}(\Omega)$ as the set of all measurable functions $u$ on $\Omega$ having a finite norm
$$
\|u\|_{L_{p,x'}L_{q,x''}(\Omega)}=\Bigg(\int_{\bR^2}\Bigg(\int_{\bR^{d-2}}\abs{u}^q \bI_{\Omega}\,dx''\Bigg)^{p/q}\,dx'\Bigg)^{1/p},
$$
where $\bI$ is the usual indicator function.
We abbreviate $L_{p, x'}L_{p,x''}(\Omega)=L_p(\Omega)$.
We will also use the notation
\begin{equation}		\label{181224@eq1}
\bR^d_{+}=\set{x=(x_1,\cdots,x_d)\in\bR^d : x_1>0},\quad B_R^{+}=\bR^d_+\cap B_R(0),
\end{equation}
$$
B'_{R}=\{x'\in \bR^2:|x'|<R\}, \quad (B'_R)^+=\bR^2_+\cap B'_R,
$$
and $\Omega_R(x)=\Omega\cap B_R(x)$ for all $x\in \bR^d$ and $R>0$.
	
Now we formulate our mixed boundary value problem. We consider domain $\Omega \subset \bR^d$ with boundary divided into two non-intersecting portions, and $\Gamma$ being the boundary of $\cD$ relative to $\p\Omega$:
$$\p\Omega=\cD\cup\cN, \quad \cD\cap \cN=\emptyset,\quad \Gamma=\p_{\p\Omega}\cD.$$
We need the following notation for Sobolev spaces with boundary conditions prescribed on the whole or part of the boundary.
For $1\le p\le \infty$, we denote by $W^{1,p}(\Omega)$ the usual Sobolev space and by $W^{1,p}_0(\Omega)$ the completion of $C^\infty_0(\Omega)$ in $W^{1,p}(\Omega)$, where $C^\infty_0(\Omega)$ is the set of all smooth, compactly supported functions in $\Omega$.
Similarly, we let $W^{1,p}_\cD(\Omega)$ be the completion of $C^\infty_\cD(\Omega)$ in $W^{1,p}(\Omega)$, where $C_\cD^\infty(\Omega)$ is the set of all smooth functions on $\overline{\Omega}$ which vanish in a neighborhood of $\cD$.

Let $L$ be a second-order elliptic operator in divergence form
$$
L u=D_i(a_{ij}(x)D_j u+b_i(x) u)+\hat{b}_i(x) D_i u+c(x)u,
$$
where the coefficients $\vec A=(a_{ij})_{i,j=1}^d$, $\vec b=(b_1,\ldots,b_d)$, $\hat{\vec b}=(\hat{b}_1,\ldots,\hat{b}_d)$, and $c$ are bounded measurable functions defined on $\overline{\Omega}$: for some positive constants $\Lambda$ and $K$, we have
$$
|\vec A|\le \Lambda^{-1}, \quad |\vec b|+|\hat{\vec b}|+|c|\le K.
$$
Note that the summation convention is adopted throughout this paper.
The leading coefficients $\vec A=(a_{ij})$ are also assumed to be symmetric, satisfy the uniformly ellipticity condition:
$$
\sum_{i,j=1}^d a_{ij}(x)\xi_j\xi_i\ge \Lambda |\xi|^2, \quad \forall \xi\in \bR^d, \quad \forall x\in \overline{\Omega}.
$$
We denote by
$$
Bu=(\vec A Du+\vec b u)\cdot n=(a_{ij}D_j u+b_i u)\,n_i
$$
the conormal derivative operator on the boundary of $\Omega$ associated with the operator $L$, where $n=(n_1,\ldots,n_d)$ is the outward unit normal to $\partial \Omega$.
We will see that in the weak formulation, this boundary condition is still well defined even when the outer unit normal is not defined point-wise.
Now we give the formal definition of weak solutions. Let $p\in (1,\infty)$.

\begin{definition}[Weak Solution]\label{def-weak-sln}
For $f, f_i\in L_p(\Omega)$, $i\in \{1,\ldots,d\}$,
we say that $u\in W^{1,p}_\cD(\Omega)$ is a weak solution to the mixed boundary value problem
\begin{equation}		\label{180417@eq2}
\begin{cases}
Lu =f+ D_i f_i  & \text{in }\, \Omega,\\
Bu =f_i  n_i  & \text{on }\, \cN,\\
u =0 & \text{on }\, \cD,
\end{cases}
\end{equation}
if
$$
\int_\Omega (- a_{ij}D_j u - b_i u) D_i \phi + (\hat{b}_i D_i u + cu)\phi\,dx=\int_\Omega f  \phi\,dx-\int_\Omega f_i D_i\phi\,dx
$$
holds for any $\phi\in W^{1,p/(p-1)}_{\cD}(\Omega)$.
\end{definition}

In this paper, we will work on the so-called Reifenberg flat domains, which is defined below in $(i)$. In $(ii)$, we assume that locally the two types of boundary conditions are almost separated: the relative boundary $\Gamma$ is also Reifenberg flat.
\begin{assumption}[$\gamma$]		\label{ass-RF}
There exists a positive constant $R_1$ such that the following hold.
\begin{enumerate}[$(i)$]
\item
For any $x_0\in \partial \Omega$ and $R\in (0,R_1]$,
there is a coordinate system depending on $x_0$ and $R$ such that in this new coordinate system (called the coordinate system associated with $(x_0, R)$), we have
\begin{equation*}	
\{y:x_{01}+\gamma R<y_1\}\cap B_R(x_0)\subset \Omega_R(x_0)\subset \{y:x_{01}-\gamma R<y_1\}\cap B_R(x_0).
\end{equation*}
\item
Let $\Gamma$ be the boundary (relative to $\partial \Omega$) of $\cD$. If $x_0\in \Gamma$ and $R\in (0, R_1]$, we can further require that the coordinate system defined in $(i)$ satisfy
$$
\Gamma \cap B_R(x_0)\subset \{y: |y'-x'|< \gamma R\}\cap B_R(x_0),
$$
$$
\big(\partial \Omega \cap B_R(x_0)\cap \{y: y_2>x_{02}+\gamma R\}\big)\subset \cD,
$$
$$
\big(\partial \Omega \cap B_R(x_0)\cap \{y: y_2<x_{02}-\gamma R\}\big)\subset \cN.
$$
\end{enumerate}
\end{assumption}
In this paper, we always assume that $\cD,\cN\neq \emptyset$,
since otherwise the boundary condition becomes purely conormal or Dirichlet. Corresponding results have been included in \cite{DK11,DK12}. See also \cite{BW1,BW2}.

We consider the equations with small ``BMO'' leading coefficients with a small parameter $\theta\in (0,1)$ to be specified later.
\begin{assumption}[$\theta$]\label{ass-smallBMO}
There exists $R_2\in (0,1]$ such that for any $x\in \overline{\Omega}$ and $r\in (0, R_2]$, we have
 $$
\dashint_{\Omega_r(x)}|a_{ij}(y)-(a_{ij})_{\Omega_r(x)}|\,dy < \theta.
 $$
\end{assumption}

In the following, we denote $R_0:= \min\set{R_1,R_2}$.

Now we can present our main result.
First, in $\Omega$ (bounded or unbounded) we consider the existence and uniqueness of $W^{1,p}_{\cD}$ weak solution to the following equation:
\begin{equation}		\label{eqn-main-large-lambda}
\begin{cases}
Lu-\lambda u= f+ D_i f_i  & \text{in }\, \Omega,\\
Bu =f_i n_i & \text{on }\, \cN,\\
u=0 &\text{on }\, \cD.
\end{cases}
\end{equation}
Compared to \eqref{180417@eq2}, here we introduce the $-\lambda u$ term to create the required decay at infinity for the unbounded domain case. For simplicity, we will use the following notation with $\lambda > 0$:
$$U:=\abs{Du}+\sqrt{\lambda}\abs{u},\quad F:=\sum_{i=1}^d\abs{f_i}+\frac{1}{\sqrt{\lambda}}\abs{f}.$$

\begin{theorem}\label{thm-well-posedness}
For any $p\in (4/3,4)$, we can find positive constants
$$
(\gamma_0,\theta_0)=(\gamma_0,\theta_0)(d,p,\Lambda), \quad \lambda_0=\lambda_0(d,p,\Lambda, R_0, K),
$$
such that the following holds. If Assumptions \ref{ass-RF} $(\gamma_0)$ and \ref{ass-smallBMO} $(\theta_0)$ are satisfied, and $\lambda > \lambda_0$, then for any $(f_i)_{i=1}^d\in (L_p(\Omega))^d$, $f\in L_p(\Omega)$ there exists a unique weak solution $u\in W^{1,p}_\cD(\Omega)$ to \eqref{eqn-main-large-lambda} satisfying
\begin{equation}\label{eqn-general}
\norm{U}_{L_p(\Omega)} \leq N \norm{F}_{L_p(\Omega)},
\end{equation}
where $N=N(d,p,\Lambda)$ is a constant.
\end{theorem}

When $\Omega$ is bounded, we have better results: instead of taking large $\lambda$, we can assume the usual sign condition $L1\le 0$, which is understood in the weak sense:
$$
\int_{\Omega} (-b_i D_i \phi+c \phi)\,dx\le 0
$$
for any $\phi\in W^{1,p/(p-1)}_{\cD}(\Omega)$ satisfying $\phi\ge 0$.
Also, the integrability of the non-divergence form source term $f$ can be generalized to $L_{p_{*}}$, where
\begin{equation}\label{eqn-def-p*}
p_{*}=\begin{cases}
pd/(p+d) &\text{when }\, p>d/(d-1),\\
1+\epsi &\text{when }\, p\leq d/(d-1)
\end{cases}
\end{equation}
for any $\epsi>0$.

\begin{theorem}\label{thm-bounded-domain}
Let $\Omega$ be a bounded domain in $\bR^d$.
For any $p\in(4/3,4)$, we can find positive constants $\gamma_0, \theta_0$ depending on $(d,p,\Lambda)$, such that the following holds. If Assumptions \ref{ass-RF} $(\gamma_0)$ and \ref{ass-smallBMO} $(\theta_0)$ are satisfied, and $L1\leq0$ in the weak sense, then for any $(f_i)_{i=1}^d\in (L_p(\Omega))^d$, $f\in L_{p_*}(\Omega)$ there exists a unique weak solution $u\in W^{1,p}_\cD(\Omega)$ to \eqref{180417@eq2} satisfying
\begin{equation}\label{eqn-a-priori-no-dependence}
\norm{u}_{W^{1,p}(\Omega)} \leq N\Bigg(\sum_{i=1}^d\norm{f_i}_{L_p(\Omega)} + \norm{f}_{L_{p_{*}}(\Omega)}\Bigg),
\end{equation}
where $N$ is a constant independent of $u,f_i$ and $f$.
\end{theorem}

In the above theorems, we always assume that
\begin{equation}		\label{190109@eq1}
p\in (4/3,4), \quad  \text{$\vec A=(a_{ij})_{i,j=1}^d$ is symmetric}.
\end{equation}
Indeed, by the Lax-Milgram Lemma and the reverse H\"older's inequality, when $p$ is close to $2$, the symmetry of $\vec A$ is not needed.
Otherwise, by the following two examples, we see that the restrictions in  \eqref{190109@eq1} are optimal for the solvability of mixed boundary value problems.
Precisely, based on a duality argument, Example \ref{eg-classical} shows the restriction $p\in (4/3,4)$ is optimal, and Example \ref{190109@ex1} shows the symmetry of $\vec A$ is required for the solvability in $W^{1,p}(\Omega)$ when $p$ is away from $2$.
Here, for the reader's convenience, we temporarily set
$$
\bR^2_+=\{x=(x_1,x_2)\in \bR^2:x_2>0\},
$$
which is different from that in \eqref{181224@eq1}.
Note that the examples below are applicable to higher dimensional cases by a trivial extension.

\begin{example}\label{eg-classical}
In $\bR^2_+$, let $u(x_1,x_2)={\rm Im}(x_1+ix_2)^{1/2}$.
One can simply check that
\begin{equation*}
\Delta u =0 \text{ in } \bR^2_{+}, \quad u=0 \text{ on } \p\bR^2_{+}\cap\set{x_1>0}, \quad \frac{\p u}{\p x_2}=0 \text{ on } \p\bR^2_{+}\cap\set{x_1<0}.
\end{equation*}
Since $Du$ is of order $r^{-1/2}$, one could also check that near the origin $Du\in L_p$ for any $p\in [1,4)$, but $Du\notin L_4$.
\end{example}

\begin{example}		\label{190109@ex1}
In $\bR^2_+$, let
$u(x_1,x_2)={\rm Im}(x_1+ix_2)^s$ with $s \in (0,1/2)$. We have
$$D_i(a_{ij}D_j u)=0 \text{ on } \bR^2_{+},\quad u=0 \text{ on } \p\bR^2_{+}\cap\set{x_1>0},\quad a_{ij}D_j u n_i=0 \text{ on }\p\bR^2_{+}\cap\set{x_1<0},
$$
where
$$(a_{ij})_{i,j=1}^{2} = \begin{bmatrix}
1 & \cot(\pi s)\\-\cot(\pi s) & 1
\end{bmatrix}.
$$
Since $Du$ is of order $r^{s-1}$, near the origin we only have $Du\in L_p$ only if $p<\frac{2}{1-s}$. Note that $\frac{2}{1-s}<4$ and $\frac{2}{1-s}\searrow 2$ as $s\searrow 0$.
\end{example}

%========================================
\section{Local Poincar\'e Inequality and Reverse H\"older's Inequality}
%========================================

In this section, we introduce two useful tools for our problem. The first one is the local Sobolev-Poincar\'e inequality. Notice that a Reifenberg flat domain intersecting with a ball might no longer be Reifenberg flat. We cannot simply localize to obtain the required local version, although Sobolev inequalities of $W^{1,p}$ hold for the Reifenberg flat domain since it is an extension domain.

\begin{theorem}[Local~Sobolev-Poincar\'e~inequality]		\label{180514@thm1}
Let $\gamma\in [0,1/48]$ and $\Omega\subset \bR^d$ be a Reifenberg flat domain satisfying Assumption \ref{ass-RF} $(\gamma)$ $(i)$.
Let $x_0\in \partial \Omega$ and $R\in (0, R_1/4]$.
Then, for any $p\in (1,d)$ and $u\in W^{1,p}(\Omega_{2R}(x_0))$, we have
$$
\|u-(u)_{\Omega_R(x_0)}\|_{L_{dp/(d-p)}(\Omega_R(x_0))}\le N\|Du\|_{L_p(\Omega_{2R}(x_0))},
$$
where $N=N(d,p)$.
\end{theorem}
\begin{proof}
See \cite[Theorem 3.5]{CDK18}.
\end{proof}

\begin{corollary}		\label{180514@cor1}
Let $\gamma\in [0,1/48]$ and $\Omega\subset \bR^d$ be a Reifenberg flat domain satisfying Assumption \ref{ass-RF} $(\gamma)$ $(i)$.
Let $x_0\in \overline{\Omega}$, $R\in (0, R_1/4]$, and $\cD\subset \partial \Omega$ with $\cD\cap B_R(x_0)\neq \emptyset$.
If there exist $z_0\in \cD\cap B_R(x_0)$ and $\alpha\in (0,1)$ such that
\begin{equation}\label{eqn-mainly-Dirichlet}
B_{\alpha R}(z_0)\subset B_R(x_0), \quad \big(\partial \Omega\cap B_{\alpha R}(z_0)\big) \subset \big(\cD \cap B_R(x_0)\big),
\end{equation}
then the following hold.
\begin{enumerate}[$(a)$]
\item
For any $p\in (1,d)$ and $u\in W^{1,p}_{\cD}(\Omega)$, we have
$$
\|u\|_{L_{dp/(d-p)}(\Omega_R(x_0))}\le N\|Du\|_{L_p(\Omega_{2R}(x_0))},
$$
where $N=N(d,p, \alpha)$.
\item
For any $u\in W^{1,2}_{\cD}(\Omega)$, we have
$$
\|u\|_{L_2(\Omega_R(x_0))}\le NR\|Du\|_{L_2(\Omega_{2R}(x_0))},
$$
where $N=N(d,\alpha)$.
\end{enumerate}
\end{corollary}

\begin{proof}
The assertion $(b)$ is a simple consequence of the assertion $(a)$.
Indeed, by taking $p\in \big(\frac{2d}{d+2}, 2\big)$, and using  H\"older's inequality and the assertion $(a)$,  we have
$$
\begin{aligned}
\|u\|_{L_2(\Omega_R(x_0))}&\le NR^{d/2-d/p+1}\|u\|_{L_{dp/(d-p)}(\Omega_R(x_0))}\\
&\le NR^{d/2-d/p+1}\|Du\|_{L_p(\Omega_{2R}(x_0))}\le N R\|Du\|_{L_2(\Omega_{2R}(x_0))}.
\end{aligned}
$$
To prove the assertion $(a)$,
we extend $u$ by zero on $B_{\alpha R}(z_0)\setminus \Omega$ so that $u\in W^{1,p}(B_{\alpha R}(z_0))$.
Since $|B_{\alpha R}(z_0)\setminus \Omega|\ge N(d)(\alpha R)^d$,
by the boundary Poincar\'e inequality, we have
\begin{equation}		\label{180514@A1}
\|u\|_{L_{dp/(d-p)}(\Omega_{\alpha R}(z_0))}\le N(d,q)\|Du\|_{L_p(\Omega_{\alpha R}(z_0))}.
\end{equation}
Notice from the triangle inequality and H\"older's inequality that
$$
\begin{aligned}
&\|u\|_{L_{dp/(d-p)}(\Omega_R(x_0))}\\
&\le \|u-(u)_{\Omega_R(x_0)}\|_{L_{dp/(d-p)}(\Omega_R(x_0))} +\|(u)_{\Omega_R(x_0)}-(u)_{\Omega_{\alpha R}(z_0)}\|_{L_{dp/(d-p)}(\Omega_R(x_0))}\\
&\quad +\|(u)_{\Omega_{\alpha R}(z_0)}\|_{L_{dp/(d-p)}(\Omega_R(x_0))}\\
&\le N\alpha^{1-d/p}\big(\|u-(u)_{\Omega_R(x_0)}\|_{L_{dp/(d-p)}(\Omega_R(x_0))} + \|u\|_{L_{dp/(d-p)}(\Omega_{\alpha R}(z_0))}\big).
\end{aligned}
$$
This combined with Theorem \ref{180514@thm1} and \eqref{180514@A1} gives the desired estimate.
\end{proof}

In the rest of the section, we shall prove the reverse H\"older's inequality for the following mixed boundary value problem without lower order terms
\begin{equation}		\label{eqn-main-no-lower}
\begin{cases}
D_i(a_{ij}D_j u)-\lambda u= f+ D_i f_i  & \text{in }\, \Omega,\\
a_{ij}D_j u n_i = f_i  n_i & \text{on }\, \cN,\\
u=0 & \text{on }\, \cD.
\end{cases}
\end{equation}
Here, we do not impose any regularity assumption (including the symmetry condition) on the coefficients $a_{ij}$.
Recall the notation that for $\lambda>0$,
$$U:=\abs{Du}+\sqrt{\lambda}\abs{u},\quad F:=\sum_{i=1}^d\abs{f_i}+\frac{1}{\sqrt{\lambda}}\abs{f}.$$

\begin{lemma}		\label{180514@lem1}
Let $\gamma\in (0,1/48]$, $\frac{2d}{d+2}<p<2$, and $\Omega\subset \bR^d$ be a Reifenberg flat domain satisfying Assumption \ref{ass-RF} $(\gamma)$.
Suppose that $u\in W^{1,2}_\cD(\Omega)$ satisfies \eqref{eqn-main-no-lower} with $f_i,f\in L_2(\Omega)$.
Let $x_0\in \overline{\Omega}$ and $R\in (0, R_1]$, satisfying either
$$
B_{R/16}(x_0)\subset \Omega \quad \text{or}\quad x_0\in \partial \Omega.
$$
Then, when $\lambda>0$, we have
$$
\int_{\Omega_{R/32}(x_0)}U^2\,dx\le NR^{d(1-2/p)}\Bigg(\int_{\Omega_{R}(x_0)}U^p\,dx\Bigg)^{2/p}+N\int_{\Omega_{R}(x_0)}F^2\,dx.
$$
When $\lambda=0$ and $f\equiv0$, we have
$$
\int_{\Omega_{R/32}(x_0)}|Du|^2\,dx\le NR^{d(1-2/p)}\Bigg(\int_{\Omega_{R}(x_0)}|Du|^p\,dx\Bigg)^{2/p}+N\int_{\Omega_{R}(x_0)}|f_i|^2\,dx.
$$
In the above, the constant $N$ depends only on $d$, $p$, and $\Lambda$.
\end{lemma}

\begin{proof}
Here we only prove for the case $\lambda>0$. When $\lambda=0$, the proof still works if we replace $U$ by $|Du|$ and $F$ by $\sum_i |f_i|$. Also, we prove only the case $x_0\in \partial \Omega$ because the  proof for the interior case is similar to the one in case (ii) for purely conormal boundary conditions.
Without loss of generality, we assume that $x_0=0$.
Let us fix $R\in (0, R_1]$.
We consider the following two cases:
$$
B_{R/16}\cap \Gamma \neq \emptyset, \quad B_{R/16}\cap \Gamma=\emptyset.
$$
\begin{enumerate}[$i)$]
\item
$B_{R/16}\cap \Gamma\neq \emptyset$.
We take $y_0\in \Gamma$ such that $\operatorname{dist}(0, \Gamma)=|y_0|$, and observe that
\begin{equation}		\label{180515@eq1a}
B_{R/16}\subset B_{R/8}(y_0)\subset B_{R/2}(y_0)\subset B_{R}.
\end{equation}
Since $u\in W^{1,2}_{\cD}(\Omega)$, as a test function to \eqref{eqn-main-no-lower}, we can use $\eta^2 u\in W^{1,2}_{\cD}(\Omega)$,
where $\eta$ is a smooth function on $\bR^d$ satisfying
$$
0\le \eta\le 1, \quad \eta\equiv 1 \,\text{ on }\, B_{R/8}(y_0), \quad \operatorname{supp} \eta \subset B_{R/4}(y_0), \quad |\nabla \eta|\le NR^{-1}.
$$
Now, using H\"older's inequality and Young's inequality, we have
\begin{equation}		\label{180515@eq1}
\int_{\Omega_{R/4}(y_0)} \eta^2 U^2\,dx\le \frac{N}{R^2}\int_{\Omega_{R/4}(y_0)} |u|^2\,dx+N\int_{\Omega_{R/4}(y_0)}F^2\,dx,
\end{equation}
where $N=N(d,\Lambda)$.
We fix a coordinate system associated with $(y_0, R/4, \Gamma)$
satisfying the properties in Assumption \ref{ass-RF} $(\gamma)$ $(ii)$.
Since we have
$$
\big(\partial \Omega \cap B_{R/4}(y_0)\cap \{y:y_2>\gamma R/4\}\big)\subset \cD,
$$
there exists $z_0\in \cD$ satisfying
$$
B_{R/16}(z_0)\subset B_{R/4}(y_0), \quad
\big(\partial \Omega \cap B_{R/16}(z_0) \big)\subset \big(\cD\cap B_{R/4}(y_0)\big).
$$
Note that because $\frac{2d}{d+2}<p<2$, we have $\frac{dp}{d-p}>2$.
Then by H\"older's inequality and Corollary \ref{180514@cor1} $(a)$, we see that
\begin{align}
\nonumber
\frac{1}{R^2}\int_{\Omega_{R/4}(y_0)}|u|^2\,dx&\le NR^{d(1-2/p)}\Bigg(\int_{\Omega_{R/4}(y_0)} |u|^{dp/(d-p)}\,dx\Bigg)^{2(d-p)/dp}\\
\label{180515@eq1b}
&\le N R^{d(1-2/p)}\Bigg(\int_{\Omega_{R/2}(y_0)}|Du|^p\,dx\Bigg)^{2/p},
\end{align}
where $N=N(d,p)$.
Combining this inequality and \eqref{180515@eq1}, and using \eqref{180515@eq1a}, we obtain the desired estimate.
	
\item
$B_{R/16}\cap \Gamma=\emptyset$.
Then $\p\Omega\cap B_{R/16}$ is contained in either $\cD$ or $\cN$. When it is in $\cD$, the proof for the previous case still works if we simply choose any $y_0\in\p\Omega\cap B_{R/16}$.
When it is contained in $\cN$, as a test function to \eqref{180417@eq2}, we can use $\zeta^2(u-c)\in W^{1,2}_{\cD}(\Omega)$,
where $c=(u)_{\Omega_{R/16}}$ and $\zeta$ is a smooth function on $\bR^d$ satisfying
$$
0\le \zeta\le 1, \quad \zeta\equiv 1 \,\text{ on }\, B_{R/32}, \quad \operatorname{supp} \zeta \subset B_{R/16}, \quad |\nabla \zeta|\le NR^{-1}.
$$
By testing \eqref{eqn-main-no-lower} with $\zeta^2 (u-c)$, we have
$$
\int_{\Omega_{R/16}} (\zeta U)^2\,dx\le \frac{N}{R^2}\int_{\Omega_{R/16}} \big|u-(u)_{\Omega_{R/16}}\big|^2\,dx+\frac{N}{R^d} \Bigg(\int_{\Omega_{R/16}} \sqrt{\lambda}|u|\,dx\Bigg)^2+N\int_{\Omega_{R/16}} F^2\,dx,
$$
where $N=N(d,\Lambda)$.
Similar to \eqref{180515@eq1b}, we get from Theorem \ref{180514@thm1} that
$$
\frac{1}{R^2}\int_{\Omega_{R/16}}\big|u-(u)_{\Omega_{R/16}}\big|^2\,dx \le NR^{d(1-2/p)}\Bigg(\int_{\Omega_{R/8}}|Du|^p\,dx\Bigg)^{2/p},
$$
where $N=N(d,p)$.
By H\"older's inequality, we also have
$$
\frac{1}{R^d}\Bigg(\int_{\Omega_{R/16}} \sqrt{\lambda}|u|\,dx\Bigg)^2\le NR^{d(1-2/p)}\Bigg(\int_{\Omega_{R/16}} \big(\sqrt{\lambda}|u|\big)^p\,dx\Bigg)^{2/p}.
$$
Combining these together, we obtain the desired estimate.
\end{enumerate}
The lemma is proved.
\end{proof}

Based on Lemma \ref{180514@lem1} and Gehring's lemma, we get the following reverse H\"older's inequality.

\begin{lemma}[Reverse H\"older's inequality]		\label{lem-reverse-holder}
Let $\gamma\in (0, 1/48]$, $p>2$, and $\Omega\subset \bR^d$ be a Reifenberg flat domain satisfying Assumption \ref{ass-RF} $(\gamma)$.
Suppose that $u\in W^{1,2}_\cD(\Omega)$ satisfies \eqref{eqn-main-no-lower} with $f_i, f\in L_p(\Omega)\cap L_2(\Omega)$.
Then there exist constants $p_0\in (2,p)$ and $N>0$, depending only on $d$, $p$, and $\Lambda$, such that for any $x_0\in \bR^d$ and $R\in (0, R_1]$, the following hold. When $\lambda>0$, we have
$$
\big(\overline{U}^{p_0}\big)^{1/p_0}_{B_{R/2}(x_0)}\le N \big(\overline{U}^{2}\big)^{1/2}_{B_{R}(x_0)}+N\big(\overline{F}^{p_0}\big)^{1/p_0}_{B_{R}(x_0)}.
$$
When $\lambda=0$ and $f\equiv 0$, we have
$$
\big(|D\overline{u}|^{p_0}\big)^{1/p_0}_{B_{R/2}(x_0)}\le N \big(|D\overline{u}|^{2}\big)^{1/2}_{B_{R}(x_0)}+N\big(|\overline{f_i}|^{p_0}\big)^{1/p_0}_{B_{R}(x_0)},
$$
where $\overline{U}$, $\overline{F}$, $D\overline{u}$, and $\overline{f_i}$ are the extensions of $U$, $F$, $Du$, and $f_i$ to $\bR^d$ so that they are zero on $\bR^d\setminus \Omega$.
\end{lemma}

\begin{proof}
Again, we only prove for the case $\lambda>0$. Let us fix a constant $p_1\in \big(\frac{2d}{d+2},2\big)$, and set
$$
\Phi=\overline{U}^{p_1}, \quad \Psi=\overline{F}^{p_1}.
$$
Then by Lemma \ref{180514@lem1}, we have
\begin{equation}		\label{180515@eq2}
\int_{B_{R/112}(x_0)} \Phi^{2/p_1}\,dx\le NR^{d(1-2/p_1)}\Bigg(\int_{B_R(x_0)} \Phi\,dx\Bigg)^{2/p_1}+N \int_{B_R(x_0)} \Psi^{2/p_1}\,dx
\end{equation}
for any $x_0\in \bR^d$ and $R\in (0, R_1]$, where $N=N(d,\Lambda, p)=N(d,\Lambda)$.
Indeed, if $B_{R/56}(x_0)\subset \Omega$, then \eqref{180515@eq2} follows from Lemma \ref{180514@lem1}.
In the case when $B_{R/56}(x_0)\cap \partial \Omega\neq \emptyset$, there exists $y_0\in \partial \Omega$ such that $|x_0-y_0|=\operatorname{dist}(x_0, \partial \Omega)$ and
$$
B_{R/112}(x_0)\subset B_{3R/112}(y_0) \subset B_{6R/7}(y_0)\subset B_{R}(x_0).
$$
Using this together with Lemma \ref{180514@lem1}, we get \eqref{180515@eq2}.
If $B_{R/56}(x_0)\subset \bR^d\setminus \Omega$, by the definition of $\overline{U}$, \eqref{180515@eq2} holds.

By \eqref{180515@eq2} and a covering argument, we have
$$
\dashint_{B_{R/2}(x_0)} \Phi^{2/p_1}\,dx\le N\Bigg(\dashint_{B_R(x_0)} \Phi\,dx\Bigg)^{2/p_1}+N \dashint_{B_R(x_0)} \Psi^{2/p_1}\,dx
$$
for any $x_0\in \bR^d$ and $R\in (0, R_1]$, where $N=N(d,\Lambda)$.
Therefore, by Gehring's lemma (see, for instance, \cite[Ch. V]{G}), we get the desired estimate.
The lemma is proved.
\end{proof}

%========================================
\section{Harmonic functions in half space with mixed boundary condition}
%========================================

In this section, we prove a regularity result for harmonic functions with mixed Dirichlet-Neumann boundary conditions on half space.
We denote
$$
B_R=B_R(0),\quad \Gamma_R^{+}:= B_R\cap \set{x_1=0,x_2>0},\quad \Gamma_R^{-}:= B_R\cap \set{x_1=0,x_2<0}.
$$

\begin{theorem}\label{thm-harmonic-mixed-halfspace}
Suppose $u \in W^{1,2}_{\Gamma^{+}_R}(B_R^{+})$ is a weak solution to
$$
\begin{cases}
\Delta u -\lambda u =0 & \text{in }B^{+}_R,\\
\frac{\p u}{\p x_1}=0 & \text{on } \Gamma_R^{-},\\
u=0 &\text{on } \Gamma_R^{+},
\end{cases}
$$
where $\lambda >0$. Then for any $p\in[2,4)$, we have $u\in W^{1,p}(B^{+}_{R/4})$ with
$$
(U^p)_{B^{+}_{R/4}}^{1/p} \leq N(d,p) (U^2)^{1/2}_{B_R^{+}}.
$$
In the case when $\lambda=0$, the same estimate holds with $|Du|$ in place of $U$.
\end{theorem}

\begin{remark}
In Theorem \ref{thm-harmonic-mixed-halfspace}, the boundary condition is only prescribed on the flat part of the boundary.
Hence the meaning of ``weak solution'' is slightly different.
In the theorem and throughout the paper,  a $W^{1,2}_{\Gamma_R^{+}}(B^{+}_R)$ weak solution means: for any $\phi\in W^{1,2}(B^{+}_R)$ satisfying $\phi=0$ on $\p B^{+}_R\setminus \Gamma_R^{-}$,
$$\int_{B_R^{+}} (\nabla u \cdot \nabla \phi + \lambda u\phi)\,dx = 0.$$
It is clear that, as a test function, one can use $\eta u$, where $\eta\in C^\infty_c(B_R)$.
\end{remark}

For the proof of  Theorem \ref{thm-harmonic-mixed-halfspace}, we will use the following two dimensional regularity result.

\begin{lemma}\label{lem-laplacian-2d}
In the half ball $B_R^{+}\subset \bR^2$, consider $u\in W^{1,2}_{\Gamma_R^+}(B_R^{+})$ which solves
\begin{equation*}
\begin{cases}
\Delta u = f  &\text{in }B^{+}_R,\\
\frac{\p u}{\p x_1}=0 &\text{on } \Gamma_R^{-},\\
u=0 &\text{on } \Gamma_R^{+},\\
\end{cases}
\end{equation*}
where $f\in L_2(B_R^{+})$. Then for any $p\in[2,4)$, we have $u \in W^{1,p}(B^{+}_{R/2})$ with
\begin{equation}\label{2destimate}
(\abs{Du}^p)_{B^{+}_{R/2}}^{1/p} \leq N(p) \big((\abs{Du}^2)^{1/2}_{B_R^{+}} + R(\abs{f}^2)^{1/2}_{B_R^{+}}\big).
\end{equation}
\end{lemma}

From the proof below, it is clear that in Lemma \ref{lem-laplacian-2d}, $R/2$ can be replaced with any $r\in (0,R)$. In this case, the constant $N$ also depends on $r$ and $R$.

\begin{proof}[Proof of Lemma \ref{lem-laplacian-2d}]
By a scaling argument, we may assume $R=1$.
We consider the following change of variables: $(y_1,y_2)\in B_{1}\cap \set{y_1>0,y_2>0}\mapsto (x_1,x_2)\in B^{+}_1$ :
$$
x_1=2y_1y_2,\quad x_2=y_2^2-y_1^2,
$$
or in complex variables:
$$x_2+ix_1=(y_2+iy_1)^2.$$
Write $\widetilde{u}(y_1,y_2)=u(x_1,x_2)$ and $\widetilde{f}(y_1,y_2)=f(x_1,x_2)$.
Then we can rewrite the equation as
$$
\begin{cases}
\Delta_y \widetilde{u}=4\abs{y}^2\widetilde{f} & \text{in } B_1^{++},\\
\frac{\p\widetilde{u}}{\p y_2}=0 & \text{on } B_{1}\cap \set{y_1>0, y_2=0},\\
\widetilde{u}=0 &\text{on } B_{1}\cap \set{y_1=0,y_2>0},
\end{cases}
$$
where $B_1^{++}:=B_1\cap \{y_1>0,y_2>0\}$.
Next, we take an even extension of $\widetilde{u}$ and $\widetilde{f}$ with respect to $y_2$-variable. Still denote the extended functions on $B_{1}^{+}$ by $\widetilde{u}$ and $\widetilde{f}$. Then the following equation is satisfied:
$$
\begin{cases}
\Delta_y \widetilde{u}=4\abs{y}^2\widetilde{f} & \text{in } B_{1}^{+},\\
\widetilde{u}=0 & \text{on } B_{1} \cap \{y_1=0\}.
\end{cases}
$$
Note that
$$
|D_x u|\le \frac{N}{|y|}|D_y \widetilde{u}|
,\quad dx=4\abs{y}^2\,dy.$$
By the Sobolev embedding theorem, the local $W^2_2$ estimate for elliptic equations, and the boundary Poincar\'e inequality, we obtain
\begin{align*}
\norm{D_y \widetilde{u}}_{L_q(B^{+}_{\sqrt 2/2})}
\le \norm{\widetilde{u}}_{W^{2,2}(B_{\sqrt 2/2}^{+})}
&\leq N \big(\norm{\widetilde{u}}_{L_2(B^{+}_{1})} + \norm{4\abs{y}^2 \widetilde{f}}_{L_2(B^{+}_{1})}\big)\\
&\leq N \big(\norm{D_y \widetilde{u}}_{L_2(B^{+}_{1})} + \norm{4\abs{y}^2 \widetilde{f}}_{L_2(B^{+}_{1})}\big)\\
&\leq N \big(\norm{D_y \widetilde{u}}_{L_2(B^{++}_{1})} + \norm{4\abs{y}^2 \widetilde{f}}_{L_2(B^{++}_{1})}\big),
\end{align*}
where $N=N(p)>0$ and $q=q(p)$ is a constant with
\begin{equation}		\label{181225@A1}
q>\frac{2p}{4-p}\ge p.
\end{equation}
Here we also used the fact that $\widetilde{u}$ and $\widetilde{f}$ are both even functions in $y_2$. Translating back to $x$-variables, we obtain
\begin{align}
\norm{\abs{x}^{\frac{q-2}{2q}}D_xu}_{L_q(B^{+}_{1/2})}
&\leq N \big(\norm{D_x u}_{L_2(B^{+}_{1})} + \norm{\abs{x}^{1/2}f}_{L_2(B^{+}_{1})}\big)\nonumber\\
&\leq N \big(\norm{D_x u}_{L_2(B^{+}_1)} + \norm{f}_{L_2(B^{+}_{1})}\big).\label{2dmidstep}
\end{align}
By H\"{o}lder's inequality and \eqref{181225@A1}, we get
\begin{align*}
\norm{D_x u}_{L_p(B^{+}_{1/2})}
&\leq \norm{\abs{x}^{\frac{q-2}{2q}}D_x u}_{L_q(B^{+}_{1/2})} \norm{\abs{x}^{-\frac{q-2}{2q}}}_{L_{qp/(q-p)}(B_{1/2}^{+})}\\
&\leq N\norm{\abs{x}^{\frac{q-2}{2q}}D_{x}u}_{L_q(B^{+}_{1/2})}.
\end{align*}
Combining this with \eqref{2dmidstep}, we obtain
$$
\norm{D_x u}_{L_p(B^{+}_{1/2})} \leq N\big(\norm{D_x u}_{L_2(B^{+}_{1})}+\norm{f}_{L_2(B^{+}_{1})}\big),
$$
which is exactly \eqref{2destimate}. The lemma is proved.
\end{proof}

We are now ready to present the proof of  Theorem \ref{thm-harmonic-mixed-halfspace}.

\begin{proof}[Proof of Theorem \ref{thm-harmonic-mixed-halfspace}]
We first prove the theorem for $\lambda=0$.
By a scaling argument and Lemma \ref{lem-laplacian-2d}, we may assume $R=1$ and $d\ge 3$.
Noting that we can differentiate both the equation and the boundary condition in $x''$-direction, the following Caccioppoli type inequality holds:
$$
\|D_x (D^k_{x''}u)\|_{L_2(B_s^+)}\le \frac{N(d,k)}{|t-s|^{k}}\|Du\|_{L_2(B_t^+)}
$$
for $0<s<t\le 1$ and $k\in \{0,1,2,\ldots\}$.
Thus by anisotropic Sobolev embedding, we can increase the integrability in $x''$-variables so that
$$
D_{x'}u\in L_{2,x'}L_{p,x''}(B^+_{r}),
\quad D_{x''}u,\,D_{x''}^2u\in L_{p}(B^+_{r})\quad \forall r<1,
$$
with the estimate
\begin{equation}		\label{181207@eq1}
\|D_{x'}u\|_{L_{2,x'}L_{p,x''}(B^+_{r})}
+\|D_{x''}u\|_{L_{p}(B^+_{r})}
+\|D_{x''}^2u\|_{L_{p}(B^+_{r})}
\le N(d,p,r)\|Du\|_{L_2(B^+_1)}.
\end{equation}
It remains to estimate $D_{x'}u$.
From \eqref{181207@eq1}, for almost every $\abs{x''}<1/2$, we have
$$
D^2_{x''}u(\cdot,x'')\in L_2\big((B'_{2/3})^+\big).
$$
Now we rewrite the equation as a $2$-dimension problem in $x'$-variables:
$$\begin{cases}
\Delta_{x'}u(\cdot,x'')= -\Delta_{x''}u(\cdot,x'') & \text{in }(B'_{2/3})^{+},\\
\frac{\p u}{\p x_1}=0 & \text{on } B'_{2/3}\cap \set{x_1=0,x_2<0},\\
u=0 &\text{on } B'_{2/3}\cap \set{x_1=0,x_2>0}.
\end{cases}$$
We apply a properly rescaled version of Lemma \ref{lem-laplacian-2d} to see that for almost every $|x''|<1/2$,
\begin{equation*}
\|D_{x'}u(\cdot,x'')\|_{L_p((B'_{1/2})^+)} \le N(p)\big(\|D_{x'}u(\cdot,x'')\|_{L_2((B'_{2/3})^+)}
+\|\Delta_{x''}u(\cdot,x'')\|_{L_2((B'_{2/3})^+)}\big).
\end{equation*}
Taking $L_p$ norm in $\{x''\in \bR^{d-2}:|x''|<1/2\}$ for both sides, and using the Minkowskii inequality and \eqref{181207@eq1} with $r=\sqrt 3/2$, we obtain $D_{x'}u\in L_p(B_{1/2}^+)$ and
$$
\|D_{x'}u\|_{L_p(B_{1/2}^+)}\le N(d,p)\|Du\|_{L_2(B_1^+)}.
$$
This gives the desired estimate for $\lambda=0$.

For a general $\lambda>0$, we use an idea by S. Agmon.
We define
$$
v(x,\tau)=u(x)\cos(\sqrt{\lambda}\tau+\pi/4),
$$
and observe that  $v$ satisfies
\begin{equation}		\label{181122@eq1}
\begin{cases}
\Delta_{(x,\tau)} v = 0 & \text{in } \hat{B}^{+}_1,\\
\frac{\p v}{\p x_1}=0 & \text{on } \hat{\Gamma}_1^{-},\\
v=0 & \text{on } \hat{\Gamma}_1^{+}.
\end{cases}
\end{equation}
where
$$
\hat{B}_1=\{(x,\tau)\in \bR^{d+1}: |(x,\tau)|<1\}, \quad \hat{B}_1^+=\hat{B}_1\cap \{x_1>0\},
$$
$$
\hat{\Gamma}_1^+=\hat{B}_1\cap \{x_1=0, x_2>0\}, \quad \hat{\Gamma}^-_1=\hat{B}_1\cap \{x_1=0, x_2<0\}.
$$
By applying the result for $\lambda=0$ to \eqref{181122@eq1}, we have
\begin{equation}	\label{181122@eq1a}
(|D_{(x,\tau)}v|^p)^{1/p}_{\hat{B}_{1/2}^+}\le N(|D_{(x,\tau)}v|^2)^{1/2}_{\hat{B}_{1}^+},
\end{equation}
where $N=N(d,p)$.
Note that the function $\Phi$ given by
$$
\Phi(\lambda)=\int_0^{1/4} \big|\cos(\sqrt{\lambda} \tau +\pi/4)\big|^p\,d\tau
$$
has a positive lower bound depending only on $p$.
Thus by using \eqref{181122@eq1a} and the fact that
$$
\big|Du(x)\cos(\sqrt{\lambda} \tau+\pi/4)\big|\le |D_{(x,\tau)}v(x,\tau)|\le U(x),
$$
we have
\begin{align*}
\int_{B_{1/4}^+} |Du|^p\,dx&\le N\int_0^{1/4} \int_{B_{1/4}^+} |Du|^p \big|\cos(\sqrt{\lambda} \tau +\pi/4)\big|^p \,dx\,d\tau \\
&\le N\int_{\hat{B}^+_{1/2}} |D_{(x,\tau)}v|^p \,dx \, d\tau\le N \Bigg(\int_{B_1^+}U^2\,dx\Bigg)^{p/2},
\end{align*}
where $N=N(d,p)$.
Similarly, from the fact that
$$
\big|\sqrt{\lambda}u(x)\sin(\sqrt{\lambda} \tau+\pi/4)\big|\le |D_{(x,\tau)}v(x,\tau)|\le U(x),
$$
we obtain
$$
\int_{B_{1/4}^+} \big|\sqrt{\lambda}u\big|^p \,dx\le N\Bigg(\int_{B_1^+}U^2\,dx\Bigg)^{p/2}.
$$
Combining these together we get the desired estimate.
The theorem is proved.
\end{proof}

%========================================
\section{Regularity of $W^{1,2}_{\cD}$ weak solutions}
%========================================

The crucial step in proving unique $W^{1,p}$ solvability is the following improved regularity result. As in Theorem \ref{thm-well-posedness}, we consider a domain $\Omega\subset \bR^d$ which can be bounded or unbounded, together with nonempty boundary portions $\cD,\cN$. Again, recall the notation that for $\lambda>0$,
$$U:=\abs{Du}+\sqrt{\lambda}\abs{u},\quad F:=\sum_{i=1}^d\abs{f_i}+\frac{1}{\sqrt{\lambda}}\abs{f}.$$

\begin{proposition}[Regularity of $W^{1,2}_{\cD}$ weak solutions]\label{prop-regularity}
For any $p\in (2,4)$, we can find positive constants $\gamma_0,\theta_0$ depending on $(d,p,\Lambda)$, such that if Assumptions \ref{ass-RF} $(\gamma_0)$ and \ref{ass-smallBMO} $(\theta_0)$ are satisfied, the following holds.
For any $W^{1,2}_{\cD}(\Omega)$ weak solution $u$ to \eqref{eqn-main-no-lower} with $\lambda>0$ and $f_i, f\in L_p(\Omega)\cap L_2(\Omega)$,
 we have $u \in W^{1,p}_{\cD}(\Omega)$ and
\begin{equation}\label{est-no-lower-order}
\norm{U}_{L_p(\Omega)} \leq N (R_0^{d(1/p-1/2)}\norm{U}_{L_2(\Omega)} + \norm{F}_{L_p(\Omega)}).
\end{equation}
Furthermore, if we also have $f\equiv 0$, then we can take $\lambda=0$, and the following estimate holds:
\begin{equation}\label{eqn-no-nondiv}
\norm{Du}_{L_p(\Omega)} \leq N(R_0^{d(1/p-1/2)}\norm{Du}_{L_2(\Omega)} + \norm{f_i}_{L_p(\Omega)}).
\end{equation}
In the above, the constant $N$  only depends on $d$, $p$, and $\Lambda$.
\end{proposition}

Based on Proposition \ref{prop-regularity}, we obtain the following
a priori estimate for the equations with lower order terms and large $\lambda$, which will be useful for the unique solvability in Theorem \ref{thm-well-posedness}.

\begin{corollary}\label{cor-no-u-rhs}
Let $p\in (2,4)$ and $\gamma_0,\theta_0$ be the constants from Proposition \ref{prop-regularity}.
Under Assumptions \ref{ass-RF} $(\gamma_0)$ and \ref{ass-smallBMO} $(\theta_0)$, there exists a positive constant $\lambda_1$ depending on $(d, p,\Lambda, R_0, K)$ such that if $u$ is a $W^{1,p}_{\cD}$ weak solution to the equation \eqref{eqn-main-large-lambda} with $f_i,f\in L_p(\Omega)$ and $\lambda>\lambda_1$, then we have
\begin{equation}		\label{181222@eq1}
\|U\|_{L_p(\Omega)}\le N\|F\|_{L_p(\Omega)},
\end{equation}
where $N=N(d,p,\Lambda)$.
\end{corollary}

The rest of this section is devoted to the proofs of Proposition \ref{prop-regularity} and Corollary \ref{cor-no-u-rhs}.

%========================================
\subsection{Decomposition of $Du$}\label{subsection-decom}
%========================================
We will use an interpolation argument to prove Proposition \ref{prop-regularity}. The key step is the following decomposition (approximation).

\begin{proposition}\label{prop-decom}
Suppose that $u\in W^{1,2}_{\cD}(\Omega)$ satisfies \eqref{eqn-main-no-lower}
with $\lambda>0$ and $f_i,f\in L_p(\Omega)\cap L_2(\Omega)$, where $p>2$.
Then under Assumptions \ref{ass-RF} $(\gamma)$ and \ref{ass-smallBMO} $(\theta)$ with $\gamma<1/(32\sqrt{d+3})$ and $\theta\in (0,1)$, for any $x_0 \in \overline{\Omega}$ and $R<R_0$, there exist positive functions $W, V \in L_2(\Omega_{R/32}(x_0))$ such that
$$U \leq W+V \quad \text{in } \Omega_{R/32}(x_0).$$
Moreover, we have for any $q < 4$,
\begin{align}
(W^2)^{1/2}_{\Omega_{R/32}(x_0)} &\leq N\big( (\theta^{\frac{1}{2\mu'}}+ \gamma^{\frac{1}{2\mu'}})(U^2)^{1/2}_{\Omega_R(x_0)} + (F^{2\mu})^{\frac{1}{2\mu}}_{\Omega_R(x_0)}\big),\label{eqn-est-W}\\
(V^q)^{1/q}_{\Omega_{R/32}(x_0)} &\leq N\big( (U^2)^{1/2}_{\Omega_R(x_0)} + (F^{2\mu})^{\frac{1}{2\mu}}_{\Omega_R(x_0)}\big).\label{eqn-est-V}
\end{align}
Here $\mu$ is a constant satisfying $2\mu=p_0$, where $p_0=p_0(d,p,\Lambda)>2$ comes from Lemma \ref{lem-reverse-holder}, and $\mu'$ satisfies $1/\mu+1/\mu'=1$.
The constant $N$ only depends on $d$, $p$, $q$, and $\Lambda$.

\end{proposition}
The rest of Section \ref{subsection-decom} will be devoted to the proof of this proposition.
\begin{proof}
According to the relative position of $x_0$ to $\cD, \cN$, we will discuss the following 3 cases.
	
{\em Case 1}: $\dist(x_0,\p\Omega)\geq R/32$.
	
In this case, our decomposition is only concerned with the interior of $\Omega$. We can do the usual ``freezing coefficient'' approximation. The existence of such $W,V$ can be found in \cite[Lemma~8.3~(i)]{DK11} or \cite[Lemma~5.1~(i)]{DK12}.

In the next two cases, we also need to approximate the Reifenberg flat boundary by hyperplane and deal with corresponding boundary conditions.

{\em Case 2}: $\dist(x_0,\p\Omega)<R/32$, $\dist(x_0,\Gamma)\geq R/24$.

In this case, either we have $B_{R/24}(x_0)\cap\cN=\emptyset$ or $B_{R/24}(x_0)\cap\cD=\emptyset$. Correspondingly, we only deal with purely Dirichlet or purely conormal boundary condition. The functions $W$ and $V$ are constructed in $\Omega_{R/24}(x_0)$. Then for the estimates, we need to shrink the radius to $R/32$. Such construction and estimates can be found in \cite[Lemma~8.3~(ii)]{DK11} and \cite[Lemma~5.1~(ii)]{DK12}. Briefly, we approximate the neighborhood of $\Omega_{R/24}(x_0)$ by a half ball thanks to the small Reifenberg flat assumption. Then we apply a cutoff technique for the Dirichlet case, or a reflection technique for the conormal case. All these two techniques will be introduced in Case 3 below.

In both Cases 1 and 2, actually we can take $q=\infty$ in \eqref{eqn-est-V}.
	
{\em Case 3}: $\dist(x_0,\p\Omega)<R/32$, $\dist(x_0,\Gamma) < R/24$.

In this case, we deal with the ``mixed'' boundary condition. Take $y_0\in \Gamma$ with $\dist(y_0, x_0)<R/24$. Consider the coordinate system associated with $(y_0,R/4)$ as in Assumption \ref{ass-RF} $(\gamma)$. For simplicity, we shift the origin in $x'=(x_1,x_2)-$hyperplane, such that
\begin{equation*}
\p\Omega\cap B_{R/4}(y_0)\subset\set{-\gamma R/2<x_1<0},
\end{equation*}
\begin{equation}\label{eqn-moving-coordinates}
\Gamma \cap B_{R/4}(y_0)\subset\set{-\gamma R/2<x_2<0}.
\end{equation}

In the following, we will omit the center when it is $y_0$. For example,
$$\Omega_{R/4}:= \Omega\cap B_{R/4}(y_0),\quad \Omega_{R/4}^{+}:=\Omega_{R/4}\cap\bR^d_{+},\quad \Omega_{R/4}^{-}:=\Omega_{R/4}\cap\bR^d_{-},
$$
where $\bR^d_-=\{x\in \bR^d:x_1<0\}$.
Note that this is slightly different from the usual convention that we omit the center when it is the coordinate origin. The following inclusion relation will be useful:
\begin{equation}
                        \label{eq11.37}
\Omega_{R/32}(x_0)\subset\Omega_{R/8}\subset\Omega_{R/2}\subset\Omega_R(x_0).
\end{equation}

Now we start to construct the decomposition. First, we introduce a cut-off function $\chi\in C^\infty(\bR^d)$ with $D\chi$ supported in a ``L-shaped'' domain, satisfying
$$\begin{cases}
\chi =0, \quad \text{on }\set{x_2>-\gamma R}\cap \set{x_1<\gamma R},\\
\chi =1, \quad \text{on }\set{x_2<-2\gamma R}\cup \set{x_1>2\gamma R},\\
0\leq \chi \leq 1, |D\chi|\leq \frac{2}{\gamma R}.
\end{cases}$$
	
The following two lemmas should be read as parts of the proof of Proposition \ref{prop-decom}. The first one is an important estimate of a typical term in our proof. Both the inequality itself and the decomposition technique in the proof will be used later.
\begin{lemma}\label{lem-usingpoincare}
We have
$$
(\abs{D\chi  u}^2)^{1/2}_{\Omega_{R/4}} \leq N \gamma^{1/(2\mu')}\big((U^2)^{1/2}_{\Omega_R(x_0)} + (F^{2\mu})^{1/(2\mu)}_{\Omega_R(x_0)}\big),
$$
where $N=N(d,p,\Lambda)$.
\end{lemma}
\begin{proof}
From the construction of $\chi$, we have
\begin{equation*}
\norm{D\chi  u}_{L_2(\Omega_{R/4})}\leq \frac{2}{\gamma R}\norm{\bI_{\supp\set{D\chi}}u}_{L_2(\Omega_{R/4})}.
\end{equation*}
Now we decompose the set $\supp\set{D\chi}\cap \Omega_{R/4}$ to obtain the required smallness. Consider the following grid points on $\p\bR^d_{+}$:
$$\cD_{grid} := \set{z\in \bR^d : z=(0,k\gamma R) \text{ for }k=(k_2,\ldots,k_d) \in \mathbb{Z}^{d-1},k_2\ge -1}\cap \Omega_{R/4}.$$
Clearly $\bigcup_{z\in\cD_{grid}}\Omega_{\sqrt {d+3}\gamma R}(z)$ covers $\supp\set{D\chi}\cap \Omega_{R/4},$ and because $\gamma<1/(32\sqrt{d+3})$
\begin{equation*}%\label{eqn-cover}
\bigcup_{z\in\cD_{grid}}\Omega_{\sqrt {d+3}\gamma R}(z)\subset \Omega_{R/3},
\end{equation*}
with each point covered by at most $N(d)$ of such neighborhoods. Due to \eqref{eqn-moving-coordinates}, we know that in each $\Omega_{\sqrt{d+3}\gamma R}(z)$, \eqref{eqn-mainly-Dirichlet} is satisfied with
$$
(z_0,\alpha)=(z+(c,(-1+\sqrt{d+3})\gamma R/2,0,\cdots,0), 1/4),
$$
where $c\in(-\gamma R/2,0)$ is chosen carefully to guarantee $z_0\in \p\Omega$. Hence we can apply the Poincar\'e inequality stated in Corollary \ref{180514@cor1} and H\"older's inequality to obtain
\begin{align}
\norm{\bI_{\supp\set{D\chi}} u}_{L_2(\Omega_{R/4})}^2
&\leq N\sum_{z\in \cD_{grid}}\norm{u}_{L_2(\Omega_{\sqrt {d+3}\gamma R}(z))}^2\nonumber \\	
&\leq N(\gamma R)^2\sum_{z\in \cD_{grid}}\norm{Du}_{L_2(\Omega_{2\sqrt {d+3}\gamma R}(z))}^2\nonumber\\
&\leq N(\gamma R)^2\norm{\bI_{|x_1|<2\sqrt {d+3}\gamma R}Du}_{L_2(\Omega_{R/3})}^2\label{eqn-overlap}\\
&\leq N(\gamma R)^2\cdot(\gamma R^d)^{1/\mu'}\norm{Du}_{L_{2\mu}(\Omega_{R/3})}^2.
\label{eqn-holder-area-suppchi-smallness}
\end{align}
To obtain \eqref{eqn-holder-area-suppchi-smallness}, we used H\"older's inequality. Now we rewrite this using the notation of average and use a properly rescaled version of Lemma \ref{lem-reverse-holder} as well as \eqref{eq11.37} to obtain
\begin{align*}
(\abs{D\chi u}^2)^{1/2}_{\Omega_{R/4}} &\leq N\gamma^{1/(2\mu')}(\abs{Du}^{2\mu})^{1/(2\mu)}_{\Omega_{R/3}}\\
&\leq N \gamma^{1/(2\mu')}\big((U^2)^{1/2}_{\Omega_R(x_0)} + (F^{2\mu})^{1/(2\mu)}_{\Omega_R(x_0)}\big).
\end{align*}
The lemma is proved.
\end{proof}

The second lemma shows how we ``freeze'' the boundary to be a hyperplane using a cut-off technique together with a reflection.
\begin{lemma}\label{lem-RF-as-pert-flat}
The function $\chi u \in W^{1,2}(\Omega_{R/4}^{+})$ satisfies the following equation in the weak sense
\begin{equation}\label{eqn-reflected}
\begin{cases}
D_i(a_{ij}D_j(\chi u)) - \lambda \chi u= D_ig_i^{(1)} + D_ig_i^{(2)} + g_i^{(3)}D_i\chi + g_i^{(4)}D_i\widetilde{\chi} + g^{(5)} & \text{in } \Omega_{R/4}^{+},\\
a_{ij}D_j(\chi u)n_i = g_i^{(1)}n_i +g_i^{(2)}n_i & \text{on } \Gamma^{-},\\
\chi u=0 &\text{on } \Gamma^{+},
\end{cases}
\end{equation}
where
\begin{align*}
g_i^{(1)}&=a_{ij}uD_j\chi + f_i\chi ,\quad g_i^{(2)}=(-\epsi_{i}\epsi_{j}\widetilde{a_{ij}}\widetilde{\chi}D_j\widetilde{u} + \epsi_{i}\widetilde{\chi}\widetilde{f_i})\bI_{(-x_1,x_2,x'')\in\Omega_{R/4}^{-}},\\
g_i^{(3)}&=a_{ij}D_ju-f_i,\quad g_i^{(4)}=(\epsi_i\epsi_j  \widetilde{a_{ij}}D_j\widetilde{u}-\epsi_i\widetilde{f_i})\bI_{(-x_1,x_2,x'')\in\Omega_{R/4}^{-}},\\
g^{(5)}&= \chi f  + \widetilde{\chi}\widetilde{f}\bI_{(-x_1,x_2,x'')\in\Omega_{R/4}^{-}}+ \lambda\widetilde{\chi}\widetilde{u}\bI_{(-x_1,x_2,x'')\in\Omega_{R/4}^{-}},\\
\Gamma^+&=\partial \Omega_{R/4}^+\cap \{x_1=0,x_2>0\}, \quad \Gamma^-=\partial \Omega_{R/4}^+\cap \{x_1=0,x_2<0\}.
\end{align*}
Here we denote $\widetilde{f}(x_1,x_2,x''):=f(-x_1,x_2,x'')$, and similarly for $\widetilde{a_{ij}}, \widetilde{\chi}$, and $\widetilde{u}$. We also use the following notation
$$\epsi_{i}:=\begin{cases}
-1& \text{if }\,\, i=1,\\
1&\text{if }\,\, i\neq 1.
\end{cases}$$
\end{lemma}
\begin{proof}
Take any test function $\psi \in W^{1,2}_{\p\Omega_{R/4}^{+}\setminus\Gamma^{-}}(\Omega_{R/4}^{+})$.
We extend $\psi$ to $\bR^d_+$ by setting $\psi\equiv 0$ on $\bR^d_+\setminus \Omega_{R/4}^+$, and then, we again extend evenly to $\bR^d$.
Denote this extended function by $\cE\psi$, and note that $\chi\cE\psi\in W^{1,2}_{\cD}(\Omega)$.
Testing \eqref{eqn-main-no-lower} with $\chi\cE\psi$ and rearranging terms will give us \eqref{eqn-reflected}.
\end{proof}

We continue the proof of Proposition \ref{prop-decom}. Solve the following equation
\begin{equation}\label{eqn-w}
\begin{cases}
\begin{aligned}
D_{i}(\overline{a_{ij}}D_j\hat{w}) - \lambda \hat{w}&=
D_i((\overline{a_{ij}}-a_{ij})D_j(\chi u)) + D_ig_i^{(1)} + D_i g_i^{(2)}\\
&\quad + g_i^{(3)}D_i\chi + g_i^{(4)}D_i\widetilde{\chi} +g^{(5)}
\end{aligned}
&\text{in }\Omega_{R/4}^{+},\\
\overline{a_{ij}}D_j\hat{w}\cdot n_i = (\overline{a_{ij}}-a_{ij})D_j (\chi u) n_i + g_i^{(1)}n_i +g_i^{(2)}n_i &\text{on }\Gamma^{-},\\
\hat{w}=0 & \text{on }\p\Omega_{R/4}^{+}\setminus\Gamma^{-},
\end{cases}
\end{equation}
for $\hat{w}\in W^{1,2}_{\p\Omega_{R/4}^{+}\setminus\Gamma^{-}}(\Omega_{R/4}^{+})$, where $\overline{a_{ij}}=(a_{ij})_{\Omega_{R/4}}$ are constants. Due to the Lax-Milgram lemma, such $\hat{w}$ exists. For simplicity, we denote
$$\hat{W}:=\abs{D\hat{w}} + \sqrt{\lambda}\abs{\hat{w}}.$$
Testing \eqref{eqn-w} by $\hat{w}$, and using the ellipticity and H\"older's inequality, we have
\begin{align}
&\norm{\hat{W}}^2_{L_2(\Omega^{+}_{R/4})} \leq \norm{(\overline{a_{ij}}-a_{ij})D_j(\chi u)}_{L_2(\Omega_{R/4}^{+})}\norm{D\hat{w}}_{L_2(\Omega^{+}_{R/4})}\label{eqn-hat-w-1}\\
&\quad+ \big\|g_i^{(1)}\big\|_{L_2(\Omega_{R/4}^{+})}\norm{D\hat{w}}_{L_2(\Omega^{+}_{R/4})} + \bignorm{g_i^{(2)}}_{L_2(\Omega_{R/4}^{+})}\norm{D\hat{w}}_{L_2(\Omega^{+}_{r/4})}\label{eqn-hat-w-2}\\
&\quad+ \bignorm{\bI_{\supp\set{D\chi}}g_i^{(3)}}_{L_2(\Omega_{R/4}^{+})}\norm{D_i\chi\hat{w}}_{L_2(\Omega_{R/4}^{+})}
+ \bignorm{g_i^{(4)}}_{L_2(\Omega_{R/4}^{+})}\bignorm{D_i\widetilde{\chi}\hat{w}\bI_{(-x_1,x_2,x'')\in\Omega_{R/4}^{-}}}_{L_2(\Omega_{R/4}^{+})}\label{eqn-hat-w-3}\\
&\quad+ \Norm{\lambda^{-1/2}{g^{(5)}}}_{L_2(\Omega_{R/4}^{+})}\norm{\lambda^{1/2}\hat{w}}_{L_2(\Omega_{R/4}^{+})}\label{eqn-hat-w-4}.
\end{align}
For the term in \eqref{eqn-hat-w-1}, we use Assumption \ref{ass-smallBMO}, H\"older's inequality, Lemma \ref{lem-reverse-holder}, and Lemma \ref{lem-usingpoincare}. Noting that $\abs{\Omega_{R/4}^{+}}$, $\abs{\Omega_{R/4}}$ and $\abs{\Omega_{R}(x_0)}$ are all comparable to $R^d$, we have:
\begin{equation}\label{eqn-est-aij}
\begin{split}
\big(\abs{(\overline{a_{ij}}-a_{ij})D_j(\chi u)}^2\big)^{1/2}_{\Omega_{R/4}^{+}}
&\leq N\big(\abs{\overline{a_{ij}}-a_{ij}}^{2\mu'}\big)^{1/(2\mu')}_{\Omega_{R/4}} \big(\abs{Du}^{2\mu}\big)^{1/(2\mu)}_{\Omega_{R/4}}
+N\big(\abs{D\chi u}^2\big)^{1/2}_{\Omega_{R/4}^{+}}\\
&\leq N\big(\abs{\overline{a_{ij}}-a_{ij}}\big)^{1/(2\mu')}_{\Omega_{R/4}}
\big(\abs{Du}^{2\mu}\big)^{1/(2\mu)}_{\Omega_{R/4}}+N\big(\abs{D\chi u}^2\big)^{1/2}_{\Omega_{R/4}^{+}}\\
&\leq N \big(\theta^{1/(2\mu')}+\gamma^{1/(2\mu')}\big)\big((U^2)^{1/2}_{\Omega_R(x_0)} + (F^{2\mu})^{1/(2\mu)}_{\Omega_R(x_0)}\big),
\end{split}
\end{equation}
where $N=N(d,p,\Lambda)$ is a constant.

For the terms in \eqref{eqn-hat-w-2}, we first estimate $g_i^{(1)}$. This is simply due to Lemma \ref{lem-usingpoincare} and H\"older's inequality:
\begin{equation}\label{eqn-est-g}
\begin{split}
\big(\bigabs{g_i^{(1)}}^2\big)^{1/2}_{\Omega_{R/4}^{+}}
&\leq N(\abs{uD\chi}^2)^{1/2}_{\Omega_{R/4}} + N(\abs{f_i}^2)^{1/2}_{\Omega_{R/4}}\\
&\leq N \gamma^{1/(2\mu')}(U^2)^{1/2}_{\Omega_R(x_0)} + N(F^{2\mu})^{\frac{1}{2\mu}}_{\Omega_R(x_0)}.
\end{split}
\end{equation}
Now we estimate $g_i^{(2)}$ as follows:
\begin{equation}\label{eqn-est-G}
\begin{split}
\big(\bigabs{g_i^{(2)}}^2\big)^{1/2}_{\Omega_{R/4}^{+}}
&\le \big(\bigabs{\widetilde{a_{ij}}\widetilde{\chi}D_j\widetilde{u}
\bI_{(-x_1,x_2,x'')\in \Omega_{R/4}^-}}^2\big)^{1/2}_{\Omega_{R/4}^{+}} +
\big(\bigabs{\widetilde{\chi}\widetilde{f_i}
\bI_{(-x_1,x_2,x'')\in \Omega_{R/4}^-}}^2\big)^{1/2}_{\Omega_{R/4}^{+}}\\
&\leq N\big(\bigabs{\bI_{\Omega_{R/4}^{-}} Du}^2\big)^{1/2}_{\Omega_{R/4}} + N(\abs{f_i}^2)^{1/2}_{\Omega_{R/4}}\\
&\leq N\gamma^{1/(2\mu')}(\abs{Du}^{2\mu})^{1/(2\mu)}_{\Omega_{R/4}} + N(\abs{f_i}^{2\mu})^{1/(2\mu)}_{\Omega_{R/4}}\\
&\leq N\gamma^{1/(2\mu')}\big((U^2)^{1/2}_{\Omega_R(x_0)} + (F^{2\mu})^{1/(2\mu)}_{\Omega_R(x_0)}\big) + N(\abs{f_i}^{2\mu})^{1/(2\mu)}_{\Omega_{R/4}},
\end{split}
\end{equation}
where in the last line, we used Lemma \ref{lem-reverse-holder}.

For the terms in \eqref{eqn-hat-w-3}, we first use the same decomposition technique together with Poincar\'e's inequality as in the proof of Lemma \ref{lem-usingpoincare} (until the step \eqref{eqn-overlap}) to obtain:
\begin{equation*}
\norm{D_i\chi\hat{w}}_{L_2(\Omega_{R/4}^{+})}\leq N \norm{D\hat{w}}_{L_2(\Omega_{R/4}^{+})}.
\end{equation*}
To avoid the problem of increased integrating domain, here we need to modify the decomposition to be $$\bigcup_{z\in\cD_{grid}}\big(\Omega_{\sqrt {d+3}\gamma R}(z)\cap\Omega_{R/4}^{+}\big),$$
and the same proof still applies. Again, with the help of H\"older's inequality and Lemma \ref{lem-reverse-holder}, we can estimate $g_i^{(3)}$ as follows:
\begin{align*}
\bignorm{\bI_{\supp\set{D\chi}}g_i^{(3)}}_{L_2(\Omega_{R/4}^{+})}
&\leq N\bignorm{\bI_{\supp\set{D\chi}}Du}_{L_2(\Omega_{R/4}^{+})} + N \bignorm{\bI_{\supp\set{D\chi}}f_i}_{L_2(\Omega_{R/4}^{+})}\\
&\leq N\gamma^{1/(2\mu')}(\norm{U}_{L_2(\Omega_R(x_0))} + R^{d/(2\mu')}\norm{F}_{L_{2\mu}(\Omega_R(x_0))}).
\end{align*}
Hence,
\begin{equation}\label{eqn-est-h}
\begin{split}
\bignorm{\bI_{\supp\set{D\chi}}&g_i^{(3)}}_{L_2(\Omega_{R/4}^{+})}\norm{D_i\chi\hat{w}}_{L_2(\Omega_{R/4}^{+})} \\
&\leq N\gamma^{1/(2\mu')}(\norm{U}_{L_2(\Omega_R(x_0))} + R^{d/(2\mu')}\norm{F}_{L_{2\mu}(\Omega_R(x_0))})
\norm{D\hat{w}}_{L_2(\Omega_{R/4}^{+})}.
\end{split}
\end{equation}
Using similar techniques as in \eqref{eqn-est-G}, we can deduce that
\begin{equation}\label{eqn-est-H}
\begin{split}
\bignorm{g_i^{(4)}}_{L_2(\Omega_{R/4}^{+})}&\bignorm{D_i\widetilde{\chi}\hat{w}\bI_{(-x_1,x_2,x'')\in\Omega_{R/4}^{-}}}_{L_2(\Omega_{R/4}^{+})}\\
&\leq N\gamma^{1/(2\mu')}(\norm{U}_{L_2(\Omega_R(x_0))} +  R^{d/(2\mu')}\norm{F}_{L_{2\mu}(\Omega_R(x_0))})
\norm{D\hat{w}}_{L_2(\Omega_{R/4}^{+})}.
\end{split}
\end{equation}
We are left to estimate the one last term in \eqref{eqn-hat-w-4}:
\begin{align}\label{eqn-est-non-div}
&\Norm{\lambda^{-1/2}{g^{(5)}}}_{L_2(\Omega_{R/4}^{+})}\nonumber\\
&\leq \Norm{\lambda^{-1/2}\chi f}_{L_2(\Omega_{R/4}^{+})} + \Norm{\lambda^{-1/2}{\widetilde{f}}
\widetilde{\chi}\bI_{(-x_1,x_2,x'')\in\Omega_{R/4}^{-}}}_{L_2(\Omega_{R/4}^{+})} +\bignorm{\lambda^{1/2}\widetilde{\chi}\widetilde{u}\bI_{(-x_1,x_2,x'')
\in\Omega_{R/4}^{-}}}_{L_2(\Omega_{R/4}^{+})}\nonumber\\
&\leq 2\norm{F}_{L_2(\Omega_{R/4})} + N\gamma^{1/(2\mu')}(\norm{U}_{L_2(\Omega_R(x_0))} + \norm{F}_{L_2(\Omega_R(x_0))}),
\end{align}
where for the last term, we applied similar techniques as we did to estimate $g_i^{(2)}$.

Substituting \eqref{eqn-est-aij}-\eqref{eqn-est-non-div} back, we obtain
\begin{equation}\label{eqn-est-w-hat-final}
(\hat{W}^2)^{1/2}_{\Omega_{R/4}^{+}}\leq N\big( (\theta^{\frac{1}{2\mu'}}+ \gamma^{\frac{1}{2\mu'}})(U^2)^{1/2}_{\Omega_R(x_0)} + (F^{2\mu})^{\frac{1}{2\mu}}_{\Omega_R(x_0)}\big).
\end{equation}
Now we define
$$W:=\begin{cases}
\hat{W} + \abs{D((1-\chi)u)} + \sqrt{\lambda}\abs{(1-\chi)u} &\text{in } \Omega_{R/32}(x_0)\cap \bR^d_{+},\\
\abs{Du} + \sqrt{\lambda}\abs{u} &\text{in } \Omega_{R/32}(x_0)\cap \bR^d_{-}.
\end{cases}$$
Using \eqref{eqn-est-w-hat-final}, H\"older's inequality, Lemma \ref{lem-reverse-holder}, and Lemma \ref{lem-usingpoincare}, we can obtain \eqref{eqn-est-W}.

To construct $V$, we set
 $$v:=\chi u-\hat{w}.$$ Clearly $v\in W^{1,2}_{\Gamma^{+}}(\Omega_{R/4}^{+})$. Simple computation using Lemma \ref{lem-RF-as-pert-flat} and \eqref{eqn-w} shows that $v$ satisfies
$$
\begin{cases}
D_{i}(\overline{a_{ij}}D_jv) - \lambda v= 0 &\text{in }\Omega_{R/4}^{+},\\
\overline{a_{ij}}D_jv \cdot n_i = 0 &\text{on }\Gamma^{-},\\
v=0 &\text{on }\Gamma^{+}.
\end{cases}
$$
Now we define
$$
V:=\begin{cases}
\abs{Dv} + \sqrt{\lambda}\abs{v}& \text{in } \Omega_{R/4}^{+},\\
0 & \text{in } \Omega_{R/4}^{-}.
\end{cases}
$$
Then we have
$$
U\le W+V \quad \text{in } \Omega_{R/32}(x_0)
$$
from the fact that
$$
u=v+\hat{w}+(1-\chi)u\quad  \text{in } \Omega_{R/32}(x_0)\cap \bR^d_+, \quad W=U \quad \text{in }\Omega_{R/32}(x_0)\cap \bR^d_-.
$$
Using \eqref{eq11.37},
we can apply a properly rescaled version of Theorem \ref{thm-harmonic-mixed-halfspace} with a change of variables to obtain that for any $q\in[2,4)$, $V\in L_q(\Omega_{R/32}(x_0))$ satisfying
\begin{equation}\label{eqn-est-v}
\begin{split}
(V^q)^{1/q}_{\Omega_{R/32}(x_0)}
\leq&
(V^q)^{1/q}_{\Omega_{R/8}^{+}}
\le N (V^2)^{1/2}_{\Omega_{R/4}^+}\\
\leq& N ((\abs{D(\chi u)}^2)^{1/2}_{\Omega_{R/4}^+} + \sqrt{\lambda}(\abs{\chi u}^2)^{1/2}_{\Omega_{R/4}^+} + \abs{\hat{W}^2}^{1/2}_{\Omega_{R/4}^+})\\
\leq& N (U^2)^{1/2}_{\Omega_R(x_0)} + N\big((\theta^{1/(2\mu')}+\gamma^{1/(2\mu')})(U^2)^{1/2}_{\Omega_R(x_0)} + (F^{2\mu})^{1/(2\mu)}_{\Omega_R(x_0)}\big)\\
\leq& N (U^2)^{1/2}_{\Omega_R(x_0)} + N(F^{2\mu})^{1/(2\mu)}_{\Omega_R(x_0)}.
\end{split}
\end{equation}
Here we used the estimates for $uD\chi$ and $\hat{W}$ in previous steps. Clearly, from \eqref{eqn-est-v} we obtain \eqref{eqn-est-V}. This finishes the proof of Proposition \ref{prop-decom}.
\end{proof}
%========================================
\subsection{Level Set Argument}
%========================================
In previous steps, we treat the perturbation problem by decomposing $U$ into two parts, with $L_2$ and $L_q$ estimates respectively. Now we interpolate using a level set argument to obtain the required $L_p$ estimate for Proposition \ref{prop-regularity}. Such argument was suggested by Caffarelli in \cite{Ca} for a ``kernel free'' approach to $W^{1,p}$ estimate of divergence form second-order elliptic equations.
Note that our estimate is not an a priori estimate, i.e., we do not need to assume $Du \in L_p$ in advance.

Define
\begin{align*}
&\cA(s):= \set{x\in\Omega:\cM_{\Omega}(U^2)^{1/2} >s},\\
&\cB(s):= \set{x\in\Omega: (\gamma^{1/(2\mu')}+\theta^{1/(2\mu')})^{-1}
\cM_{\Omega}(F^{2\mu})^{1/(2\mu)} + \cM_{\Omega}(U^2)^{1/2} >s },
\end{align*}
where $\mu,\mu'\in (1,\infty)$ are the constants from Proposition \ref{prop-decom}.
Here we denote $\cM_{\Omega}$ to be the Hardy-Littlewood maximal  operator restricted on $\Omega$, i.e., for $f\in L_{1,\rm{loc}}(\Omega)$ and $x\in\Omega$:
$$
\cM_{\Omega}(f)(x):=\sup_{r>0}\fint_{B_r(x)}|f|
\bI_{\Omega}.
$$
By the Hardy-Littlewood  theorem, for any $f\in L_q(\Omega)$ with $q\in [1,\infty)$, we have
\begin{equation}\label{eqn-weak-est}
\abs{\set{x\in\Omega:\cM_{\Omega}(f)(x)>s}} \leq N\frac{\norm{f}_{L_q(\Omega)}^q}{s^q},
\end{equation}
where $N=N(d,q)$.
	
Proposition \ref{prop-decom} leads to the following lemma.

\begin{lemma}\label{lem-levelset}
Under the same hypothesis of Proposition \ref{prop-decom}, for any $q\in[2,4)$, there exists a constant $N$ depending on $(d,p,q,\Lambda)$, such that for all $\kappa>2^{d/2}$ and $s>0$, the following holds: if for some $R<R_0, x_0\in \overline{\Omega}$,
\begin{equation}\label{eqn-stoptime}
\abs{\Omega_{R/128}(x_0)\cap\cA(\kappa s)} \geq N\big(\kappa^{-q} + \kappa^{-2}(\gamma^{1/\mu'}+\theta^{1/\mu'})\big)\abs{\Omega_{R/128}(x_0)},
\end{equation}
then $\Omega_{R/128}(x_0)\subset \cB(s)$.
\end{lemma}
\begin{proof}
Without loss of generality, we assume $s=1$. We also extend $U$ and $F$ to be zero outside $\Omega$. We will prove the contrapositive of the above statement.
	
Suppose there exists a point $z_0$, with $$z_0\in \Omega_{R/128}(x_0), \quad z_0\notin\cB(1),$$ then by the definition of $\cB$, we have
$$
(\gamma^{1/(2\mu')}+\theta^{1/(2\mu')})^{-1}
\cM_{\Omega}(F^{2\mu})^{1/(2\mu)}(z_0) + \cM_{\Omega}(U^2)^{1/2}(z_0) \leq 1.
$$
In particular, for any $r>0$, we have
\begin{equation*}
(\gamma^{1/(2\mu')}+\theta^{1/(2\mu')})^{-1}
(F^{2\mu})^{1/(2\mu)}_{B_r(z_0)} + (U^2)^{1/2}_{B_r(z_0)} \leq 1.
\end{equation*}
Using Proposition \ref{prop-decom} with $z_0$ in place of $x_0$, we can find $W,V$ defined on $\Omega_{R/32}(z_0)$, such that for any $q\in [2,4)$,
\begin{equation}\label{eqn-WV-bound}
\begin{split}
&U\leq V+W \quad \text{in } \Omega_{R/32}(z_0),\\
&(W^2)^{1/2}_{\Omega_{R/32}(z_0)}\leq N(\gamma^{1/(2\mu')}+\theta^{1/(2\mu')}),\quad (V^q)^{1/q}_{\Omega_{R/32}(z_0)}\leq N.
\end{split}
\end{equation}
Notice that we have the following inclusion
\begin{equation}\label{eqn-inclusion}
{\Omega_{R/128}(x_0)\subset\Omega_{R/64}(z_0)\subset\Omega_{R/32}(z_0)}.
\end{equation}
Now for any $y_0\in \Omega_{R/128}(x_0)\cap\cA(\kappa)$, by the definition of $\cA$, we can find some $r>0$ such that
\begin{equation*}
\Bigg(\fint_{B_r(y_0)}U^2\,dx\Bigg)^{1/2}>\kappa.
\end{equation*}
We claim that $r<R/64$. Otherwise noting $y_0\in \Omega_{R/64}(z_0)$, we have $\Omega_r(y_0)\subset\Omega_{2r}(z_0)$.
Hence we can deduce that
\begin{align*}
\Bigg(\fint_{B_r(y_0)}U^2\,dx\Bigg)^{1/2}\leq 2^{d/2}\Bigg(\fint_{B_{2r}(z_0)}U^2\,dx\Bigg)^{1/2}\leq 2^{d/2} \cM_{\Omega}(U^2)^{1/2}(z_0)\leq 2^{d/2}< \kappa,
\end{align*}
which is a contradiction.

Now, since $r<R/64$,  the decomposition $U\leq W+V$ is defined in $\Omega_r(y_0)\subset\Omega_{R/32}(z_0)$. Extending $W$ and $V$ to be zero outside $\Omega$, we have
\begin{align*}
\Bigg(\fint_{B_r(y_0)}U^2\,dx\Bigg)^{1/2}
&\leq \Bigg(\fint_{B_r(y_0)}W^2\,dx\Bigg)^{1/2}+\Bigg(\fint_{B_r(y_0)}V^2\,dx\Bigg)^{1/2}\\
&\leq \cM_{\Omega}(W^2 \bI_{\Omega_{R/32}(z_0)})^{1/2}(y_0)+ \cM_{\Omega}(V^2 \bI_{\Omega_{R/32}(z_0)})^{1/2}(y_0).
\end{align*}
Then by \eqref{eqn-weak-est}, \eqref{eqn-WV-bound} and \eqref{eqn-inclusion}, we obtain
\begin{align*}
\abs{\Omega_{R/128}(x_0)\cap\cA(\kappa)}
\leq& \abs{\Omega_{R/32}(z_0)\cap\cA(\kappa)}\\
\leq&
\bigabs{\set{\cM_{\Omega}(W^2 \bI_{\Omega_{R/32}(z_0)})^{1/2}>\kappa/2}} + \bigabs{\set{\cM_{\Omega}(V^2 \bI_{\Omega_{R/32}(z_0)})^{1/2}>\kappa/2}}\\
\leq& N\frac{\norm{W}^2_{L_2(\Omega_{R/32}(z_0))}}{(\kappa/2)^2} + N\frac{\norm{V}^q_{L_q(\Omega_{R/32}(z_0))}}{(\kappa/2)^q}\\
\leq& N\abs{\Omega_{R/32}(z_0)}\big(\kappa^{-2}(\gamma^{1/\mu'}+\theta^{1/\mu'})+\kappa^{-q}\big)\\
\leq& N \big(\kappa^{-2}(\gamma^{1/\mu'}+\theta^{1/\mu'})+\kappa^{-q}\big)\abs{\Omega_{R/128}(x_0)}.
\end{align*}
Here $N=N(d,p ,q,\Lambda)$ is exactly what we aim to find.
\end{proof}

Using a lemma in measure theory called ``crawling of the ink spot'' which was first introduced by Krylov and Safonov in \cite{KS,S}, we obtain the following decay estimate from Lemma \ref{lem-levelset}.
\begin{corollary}[Decay of $\cA(s)$]\label{cor-decay}
Under the same hypothesis of Proposition \ref{prop-decom}, for any $q\in [2,4)$, there exists a constant $N$ depending on $(d,p,q,\Lambda)$, such that for any $\kappa>\max\set{2^{d/2},\kappa_0}$ and
\begin{equation}\label{eqn-s0}
s>s_0(d,p,q,\Lambda,\kappa,R_0,\norm{U}_{L_2(\Omega)}):=\Bigg(\frac{\norm{U}_{L_2(\Omega)}^2}{N\kappa^2(\kappa^{-q} + \kappa^{-2}(\gamma^{1/\mu'}+\theta^{1/\mu'}))|B_{R_0/128}|}\Bigg)^{1/2},
\end{equation}
we have
%\begin{equation}\label{eqn-result-crawling-of-ink}
$$
\abs{\cA(\kappa s)} \leq N\big(\kappa^{-q} + \kappa^{-2}(\gamma^{1/\mu'}+\theta^{1/\mu'})\big)\abs{\cB(s)},
$$
%\end{equation}
where $\kappa_0$ is the constant satisfying
\begin{equation}\label{eqn-kappa}
N\big(\kappa_0^{-q} + \kappa_0^{-2}(\gamma^{1/\mu'}+\theta^{1/\mu'})\big) <1/3.
\end{equation}
\end{corollary}
Here we only sketch the proof. The key idea is to use a stopping time argument (or the Calder\'on-Zygmund decomposition as in \cite{Ca}). Different from Krylov and Safonov's original version, we cover $\Omega$ by balls instead of dyadic cubes. For any $x_0\in\cA(\kappa s)$, by \eqref{eqn-weak-est}, \eqref{eqn-s0}, and \eqref{eqn-kappa}, we see that \eqref{eqn-stoptime} does not hold with $R_0$ in place of $R$. We shrink the ``ball'' $\Omega_{R/128}(x_0)$ from $R=R_0$ until the first time \eqref{eqn-stoptime} holds. Due to \eqref{eqn-s0}, \eqref{eqn-kappa} and the Lebesgue differentiation theorem, such $R$ exists and $R\in (0,R_0)$. We are left to use the Vitali covering lemma to pick a ``almost disjoint'' cover.

%========================================
\subsection{Proof of Proposition \ref{prop-regularity} and Corollary \ref{cor-no-u-rhs}}
%========================================
Now we are ready to give the proof of Proposition \ref{prop-regularity}.
\begin{proof}[Proof of Proposition \ref{prop-regularity}]
Let us fix $p\in (2,4)$, and let $\gamma$, $\theta$, and $\kappa$ be positive constants to be chosen later, such that
$$
\gamma<1/(32\sqrt{d+3}), \quad \theta<1, \quad \kappa >\max\{2^{d/2},\kappa_0\},
$$
where $\kappa_0=\kappa_0(d,p,\Lambda)$ is a constant satisfying \eqref{eqn-kappa} with $q=(p+4)/2$.
It suffices to prove
\begin{equation}\label{eqn-rewrite-Lp}
\lim_{S\rightarrow \infty} \int_0^Sp\abs{\cA(s)}s^{p-1}\,ds \leq N\big(R_0^{d(1-p/2)}\norm{U}_{L_2(\Omega)}^p + \norm{F}_{L_p(\Omega)}^p\big)
\end{equation}
under Assumptions \ref{ass-RF} $(\gamma)$ and \ref{ass-smallBMO} $(\theta)$.
The left-hand side becomes
%\begin{equation}\label{eqn-cov}
$$
\lim_{S\rightarrow \infty} \int_0^{S/\kappa}p\kappa^p\abs{\cA(\kappa s)}s^{p-1}\,ds.
$$
%\end{equation}
We bound the integrand by using Corollary \ref{cor-decay} with $q=(p+4)/2$ when $s>s_0$. For $s\leq s_0$, we apply Chebyshev's inequality.
Then we have
\begin{align*}
&\int_0^{S/\kappa} \abs{\cA(\kappa s)}\kappa^ps^{p-1}\,ds \\
&\leq N\int_0^{s_0}\frac{\norm{U}^2_{L_2(\Omega)}}{(\kappa s)^2}\kappa^ps^{p-1}\,ds + N\big(\kappa^{-q} + \kappa^{-2}(\gamma^{\frac 1 {\mu'}}+\theta^{\frac 1 {\mu'}})\big) \int_0^{S/\kappa}\abs{\cB(s)}\kappa^p s^{p-1}\,ds\\
&\leq N_{0}R_0^{d(1-p/2)}\norm{U}_{L_2(\Omega)}^{p}+ N_{1}\big(\kappa^{-q} + \kappa^{-2}(\gamma^{\frac 1 {\mu'}}+\theta^{\frac 1 {\mu'}})\big)\kappa^p \int_0^{S/\kappa}\abs{\cA(s/2)} s^{p-1}\,ds+ N\norm{F}_{L_p(\Omega)}^p,
\end{align*}
where $N_1=N_1(d,p,\Lambda)$ and $N_0$ depends also on $\kappa$.
Here in the last line, we used the following relationship:
\begin{equation*}
\cB(s)\subset\cA(s/2)\cup\set{(\gamma^{1/(2\mu')}+\theta^{1/(2\mu')})^{-1}
\cM_{\Omega}(F^{2\mu})^{1/(2\mu)} >s/2}
\end{equation*}
and the Hardy-Littlewood inequality, noting that $2\mu<p$.
Now we choose $\kappa$ sufficient large such that $N_{1}\kappa^{p-q}<2^{-p-2}$, and then $\theta$ and $\gamma$ sufficient small such that $N_{1}(\kappa^{p-2}(\gamma^{1/\mu'}+\theta^{1/\mu'}))<2^{-p-2}$. Then we have
\begin{equation*}
\int_0^Sp\abs{\cA(s)}s^{p-1}\,ds
\leq  NR_0^{d(1-p/2)}\norm{U}_{L_2(\Omega)}^{p} + N\norm{F}_{L_p(\Omega)}^p+ \frac{p}{2} \int_0^{S/(2\kappa)}\abs{\cA(s)}s^{p-1}\,ds,
\end{equation*}
where $N=N(d,p,\Lambda)$.
This yields \eqref{eqn-rewrite-Lp}. Hence, we have $u \in W^{1,p}(\Omega)$ satisfying \eqref{est-no-lower-order}. Note that all the previous proof including the reverse H\"older inequality, the estimate for harmonic functions, the decomposition lemma, and the level set argument, also work when $\lambda=0$ if we substitute $U$ by $|Du|$ and $F$ by $\sum_i|f_i|$. Thus we can also obtain \eqref{eqn-no-nondiv} when $\lambda=0$ and $f=0$.
\end{proof}

To end this section, we give the proof of Corollary \ref{cor-no-u-rhs}.

\begin{proof}[Proof of Corollary \ref{cor-no-u-rhs}]
Consider the usual smooth cut-off function $\zeta\in C^\infty_c(B_{\epsi})$ with $\zeta \in [0,1], \abs{D\zeta}\le N/\epsi$. From \eqref{eqn-main-large-lambda}, we can obtain the equation for $\zeta u$:
\begin{equation*}
\begin{cases}
D_i(a_{ij}D_j(u\zeta))-\lambda (u\zeta) = D_i(f_i\zeta + h_i)+ (f\zeta + h) & \text{in }\, \Omega,\\
a_{ij}D_j (u\zeta)  n_i =(f_i\zeta + h_i)  n_i &\text{on }\, \cN,\\
u\zeta=0 & \text{on }\, \cD,
\end{cases}
\end{equation*}
where $h_i$ and $h$ are given as follows:
\begin{equation*}
\begin{split}
h_i&:= a_{ij}uD_j\zeta - b_i u\zeta,\\
h&:=a_{ij}D_juD_i\zeta + b_iuD_i\zeta - \hat{b}_iD_iu\zeta-cu\zeta-f_iD_i \zeta.
\end{split}
\end{equation*}
Hence by \eqref{est-no-lower-order}, for any $\lambda>0$, we have
\begin{equation}		\label{190112@eq3}
\begin{aligned}
&\norm{D(u\zeta)}_{L_p(\Omega)}+\sqrt{\lambda}\norm{u\zeta}_{L_p(\Omega)}\\
&\leq N_0 R_0^{d(1/p-1/2)}\big(\norm{D(u\zeta)}_{L_2(\Omega)} + \sqrt{\lambda}\norm{u\zeta}_{L_2(\Omega)}\big)\\
&\quad+N_0 \|f_i \zeta\|_{L_p(\Omega)}+\frac{N_0}{\sqrt{\lambda}}\big(\|f_i D\zeta\|_{L_p(\Omega)}+\|f\zeta\|_{L_p(\Omega)}\big)\\
& \quad +N_1\big(\norm{uD\zeta}_{L_p(\Omega)}+\norm{u\zeta}_{L_p(\Omega)}\big)\\
&\quad+\frac{N_1}{\sqrt{\lambda}}\big(\norm{Du\cdot D\zeta}_{L_p(\Omega)}+\norm{uD\zeta}_{L_p(\Omega)}+\norm{Du\zeta}_{L_p(\Omega)}+\|u\zeta\|_{L_p(\Omega)}\big),
\end{aligned}
\end{equation}
where $N_0=N_0(d,p, \Lambda)$ and $N_1$ depends also on $K$.
Using H\"older's inequality, we obtain
\begin{equation*}
\norm{D(u\zeta)}_{L_2(\Omega)} + \sqrt{\lambda}\norm{u\zeta}_{L_2(\Omega)} \leq \epsi^{d/2-d/p}\big(\norm{D(u\zeta)}_{L_p(\Omega)} + \sqrt{\lambda}\norm{u\zeta}_{L_p(\Omega)}\big).
\end{equation*}
Thus by taking $\varepsilon=\varepsilon(d,p,\Lambda, R_0)>0$ sufficiently small  such that $N_0(\varepsilon/R_0)^{d/2-d/p}<1/2$, we can absorb the first two terms on the right-hand side of \eqref{190112@eq3} to the left-hand side.
Then by using the standard partition of unity technique and choosing $\lambda$ large enough, we conclude \eqref{181222@eq1}.
The corollary is proved.
\end{proof}

%========================================
\section{Solvability and General $p$}
%========================================
With the regularity result in hand, we are now going to prove Theorem \ref{thm-well-posedness} concerning the solvability. Note that in this section, we deal with more general cases $p\in(4/3,4)$. We first state the following $L_2$ well-posedness result, which is a direct consequence of the Lax-Milgram lemma.

\begin{lemma}\label{lem-L2-solvability}
Let $\Omega$ be a domain with $\p\Omega=\cD\cup\cN$. Consider the equation \eqref{eqn-main-large-lambda} with $b_i,\hat{b}_i,c\in L_\infty(\Omega)$. Then for any
$$f,f_i\in L_2(\Omega),\quad \lambda > \lambda_2:= 4\Big(\norm{b_i}^2_{L_\infty(\Omega)}/\Lambda
+\norm{\hat{b}_i}^2_{L_\infty(\Omega)}/\Lambda+\norm{c}_{L_\infty(\Omega)}\Big),$$
there exists a unique $W^{1,2}_{\cD}(\Omega)$ weak solution $u$ to \eqref{eqn-main-large-lambda}, satisfying
\begin{equation*}
\norm{U}_{L_2(\Omega)}\leq N \norm{F}_{L_2(\Omega)},
\end{equation*}
where $N=N(\Lambda)$ is a constant.
\end{lemma}

\begin{proof}[Proof of Theorem \ref{thm-well-posedness}]
We prove by three cases under Assumptions \ref{ass-RF} $(\gamma_0)$ and \ref{ass-smallBMO} $(\theta_0)$, where $\gamma_0, \theta_0$ are the constants from Proposition \ref{prop-regularity}.
Assume that  $\lambda>\lambda_0$, where $\lambda_0$ is a constant to be chosen below, which satisfies
$$
\lambda_0\ge\max\{\lambda_1, \lambda_2\}.
$$
Here, $\lambda_1$ and $\lambda_2$ are the constants from Corollary \ref{cor-no-u-rhs} and Lemma \ref{lem-L2-solvability}, respectively.

{\em Case 1}: $p=2$. This is Lemma \ref{lem-L2-solvability}.

{\em Case 2}: $p\in(2,4)$.
Due to the method of continuity and the a priori estimate (see Corollary \ref{cor-no-u-rhs}), it suffices to prove the theorem when all the lower order coefficients are zero, i.e., $b_i\equiv \hat{b}_i\equiv c\equiv 0$.

Now we approximate $f,f_i$ by $f^{(n)}, f^{(n)}_i$ strongly in $L_p(\Omega)$, where $f^{(n)}, f^{(n)}_i\in L_2(\Omega)\cap L_p(\Omega)$.
Then by Lemma \ref{lem-L2-solvability}, there exist $W^{1,2}_{\cD}(\Omega)$ weak solutions $u^{(n)}$ to \eqref{eqn-main-large-lambda} (without lower order terms) with $f^{(n)}, f_i^{(n)}$ in place of $f, f_i$.
Moreover, it follows from Proposition \ref{prop-regularity} and Corollary \ref{cor-no-u-rhs} that $\{u^{(n)}\}$ is a Cauchy sequence in $W^{1,p}_{\cD}(\Omega)$.
Denote its limit by $u\in W^{1,p}_{\cD}(\Omega)$. Clearly $u$ is a weak solution, and \eqref{eqn-general} is satisfied.

{\em Case 3}: $p\in(4/3,2)$.
We first use a duality argument to prove the a priori estimate, i.e. assuming $u\in W^{1,p}_{\cD}(\Omega)$ is a solution to \eqref{eqn-main-large-lambda} with $f,f_i\in L_p(\Omega)$, we are to prove the estimate \eqref{eqn-general}. For simplicity, we consider the following equivalent norm for the space $L_p(\Omega)\times (L_p(\Omega))^d$ with $\lambda>0$:
$$\norm{(f,(f_i)_{i=1}^d)}_{p,\lambda}:= \lambda^{-1/2}\norm{f}_{L_p(\Omega)} + \sum_{i=1}^d \norm{f_i}_{L_p(\Omega)},$$
and its dual space $L_{p'}(\Omega)\times (L_{p'}(\Omega))^d$:
$$\norm{(f,(f_i)_{i=1}^d)}_{p',1/\lambda}:= \lambda^{1/2}\norm{f}_{L_{p'}(\Omega)} + \sum_{i=1}^d \norm{f_i}_{L_{p'}(\Omega)},$$
where $p'\in(2,4)$ satisfying $1/p+1/p'=1$.

By duality, to prove \eqref{eqn-general}, it suffices to prove
\begin{equation*}
\sup_{\substack{\varphi_i,\varphi \in C^\infty_c(\Omega)\\
\norm{(\varphi,(\varphi_i)_{i=1}^d)}_{p',1/\lambda}=1}}\Abs{\int_\Omega (D_i u \varphi_i + \lambda u \varphi )\, dx} \leq N\left(\norm{f_i}_{L_p(\Omega)}+ \lambda^{-1/2}\norm{f}_{L_p(\Omega)}\right).
\end{equation*}
For this, we solve for $v\in W^{1,p'}_{\cD}(\Omega)$ to the following adjoint problem:
\begin{equation}\label{eqn-dual}
\begin{cases}
D_i(a_{ji}D_j v - \hat{b}_i v) - b_i D_i v + cv-\lambda v=\lambda \varphi - D_i \varphi_i  & \text{in }\, \Omega,\\
a_{ji}D_j v n_i - \hat{b}_i v n_i=-\varphi_i \cdot n_i & \text{on }\, \cN,\\
v=0 & \text{on }\, \cD.
\end{cases}
\end{equation}
Noting that $u$ is a test function for \eqref{eqn-dual}, and $v$ is a test function for \eqref{eqn-main-large-lambda}, we have
\begin{align*}
\Abs{\int_\Omega (D_i u \varphi_i + \lambda u \varphi)\,dx} &= \Abs{\int_\Omega (-a_{ji}D_j v D_iu + \hat{b}_iv D_iu - b_iD_ivu+cvu -\lambda vu)\,dx}\\
&=\Abs{\int_\Omega (-f_iD_iv + fv)\,dx}\\
&\leq \big(\|f_i\|_{L_p(\Omega)}+\lambda^{-1/2}\|f\|_{L_p(\Omega)}\big)\big(\|Dv\|_{L_{p'}(\Omega)}+\lambda^{1/2}\|v\|_{L_{p'}(\Omega)}\big)\nonumber
\\
&= N\big(\|f_i\|_{L_p(\Omega)}+\lambda^{-1/2} \|f\|_{L_p(\Omega)}\big).
\end{align*}
Here, we use the following $W^{1,p'}$ estimate for the equation \eqref{eqn-dual}:
\begin{equation*}
\norm{D_iv}_{L_{p'}(\Omega)} + \sqrt{\lambda}\norm{v}_{L_{p'}(\Omega)} \leq N\big(\norm{\varphi_i}_{L_{p'}(\Omega)}
+\lambda^{-1/2}\norm{\lambda\varphi}_{L_{p'}(\Omega)}\big)=N.
\end{equation*}
This gives us the $W^{1,p}$-a priori estimate, i.e., \eqref{eqn-general} when $p\in(4/3,4)$.

To see the solvability, we approximate $f,f_i$ by
$$f^{(n)},f^{(n)}_i \in C^\infty_c (\subset L_p), \quad f^{(n)}\rightarrow f, \ f^{(n)}_i\rightarrow f_i \quad \text{in } L_p.$$
Let $u^{(n)}$ be the unique $W^{1,2}_\cD$ weak solution associated with $f^{(n)}$ and $f^{(n)}_i$. Due to the $W^{1,p}$-a priori estimate that we just obtained, and the same argument as in case 2, it is enough to show $u^{(n)}\in W^{1,p}_\cD(\Omega)$. Due to H\"older's inequality, this can be further reduced to showing the following:
\begin{equation}\label{eqn-holder-decay}
\sum_k\norm{u}_{W^{1,2}(\Omega_{k+1}\setminus\Omega_{k})}\cdot k^{(d-1)(1/p-1/2)}<\infty.
\end{equation}
Here, we denoted $\Omega_k:=\Omega_k(0)$. For this, we use a classical ``hole-filling'' technique. Take $\eta\in C^\infty_c(B_k^c)$, $\eta =1$ in $B_{k+1}^c$, $\abs{D\eta}\leq 2$. Testing the equation by $u\eta^2$ and rearranging terms, we obtain that there exits some $\lambda_0=\lambda_0(d,p,\Lambda,R_0,\norm{b_i}_\infty,\norm{\hat{b_i}}_\infty,\norm{c}_\infty)$, such that for $\lambda>\lambda_0$,
\begin{equation*}
\int_{\Omega_{k+1}^c} (\abs{Du}^2 + \lambda\abs{u}^2)\,dx \leq N \int_{\Omega_{k+1}\setminus\Omega_{k}} \abs{u}^2\,dx.
\end{equation*}
Clearly, this leads to
\begin{equation*}
\norm{Du}^2_{L_2(\Omega_{k+1}^c)} + \lambda\norm{u}^2_{L_2(\Omega_{k+1}^c)} \leq \frac{N}{N+\lambda} (\norm{Du}^2_{L_2(\Omega_{k}^c)} + \sqrt{\lambda}\norm{u}^2_{L_2(\Omega_{k}^c)}).
\end{equation*}
Hence, $\norm{u}_{W^{1,2}(\Omega_{k+1}\setminus\Omega_{k})}$ decays exponentially and in particular, \eqref{eqn-holder-decay} holds. This finishes our proof.
\end{proof}
%========================================
\section{Bounded Domain Case}
%========================================
In this section, we deal with the bounded domain case, i.e., Theorem \ref{thm-bounded-domain}. First, we reduce the problem to the case $f=0$ by solving a divergence equation. This reduction has also been used in \cite{CDK18}. Note that in the following, we use a key fact that a Reifenberg flat domain is also a so-called John domain, which can be found in \cite[Remark~3.3]{DK18-JDE}.

Let us first recall the definition of John domains.
\begin{definition}[John domain, \cite{ADM}]\label{def-John}
A bounded set $\Omega\subset\bR^d$ is a John domain, if there exist $x_0\in \Omega$ and $\lambda>0$ such that for every $x\in\Omega$ there exists a continuous rectifiable curve $\gamma: [0,1]\mapsto \Omega$, such that $\gamma(0)=x,\gamma(1)=x_0$, and
\begin{equation}\label{eqn-John}
\dist(\gamma(t),\Omega^c) \geq \lambda \cdot\abs{\gamma[0,t]}
\end{equation}
for all $t\in[0,1]$, where $\abs{\gamma[0,t]}$ represents the arc length.
\end{definition}

\begin{lemma}\label{lem-solve-divergence-eqn}
Assume $\Omega$ is a bounded Reifenberg flat domain with $\p\Omega=\cD\cup\cN, \cD, \cN \neq \emptyset$, satisfying Assumption \ref{ass-RF} $(\gamma)$, $\gamma<1/2$.
Let $p>1$ and $p_*$ be given as in \eqref{eqn-def-p*}.
Then for every $f\in L_{p_{*}}(\Omega)$, there exists $\phi=(\phi_1,\cdots,\phi_d)\in (W^{1,p_*}_\cN(\Omega))^d$ ($\subset(L_p(\Omega))^d$, by the Sobolev embedding), such that
\begin{equation}\label{eqn-est-div}
D_i \phi_i = f \text{ in } \Omega, \quad \norm{\phi_i}_{L_p(\Omega)}\leq N\norm{f}_{L_{p_{*}}(\Omega)},
\end{equation}
where $N=N(d,p,\diam(\Omega),R_1)$ is a constant.
\end{lemma}
\begin{proof}
Noting that $\cD,\cN\neq\emptyset$, we can choose a point $x_0\in\Gamma$. Taking the coordinate system in $B_{R_1}(x_0)$ from Assumption \ref{ass-RF}, we extend $\Omega$ beyond $\cD$ as follows.

We first take the Whitney decomposition of the open set
$$
\widetilde{\Omega}_{R_1}:=\Omega_{R_1}(x_0)\cap\set{x_{02}+23/32R_1 < y_2 < x_{02}+25/32R_1}
$$
as in \cite[Chapter IV]{St}, i.e., $\widetilde{\Omega}_{R_1}=\cup_k Q_k$, where the disjoint cubes $Q_k$ satisfy
$$
\diam(Q_k)\leq \dist(\overline{Q_k},(\widetilde{\Omega}_{R_1})^c) \leq 4\diam(Q_k).
$$
Denote the center of $Q_k$ to be $x_k$. We extend $Q_k$ to $\hat{Q}_k$ in the way that $$\hat{Q}_k - x_k = 8(Q_k - x_k).
$$
Let $\hat{\Omega}=\Omega\cup(\cup_k \hat{Q}_k)$. It is easy to see that
$$
\cN \subset \p\hat{\Omega}, \quad C_1 R_1^d\leq\abs{\hat{\Omega}\setminus\Omega}\leq C_2 R_1^d,
$$
where $C_1, C_2$ are constants only depending on the space dimension $d$. Next, we check that $\hat{\Omega}$ is still a John domain, i.e., for any $\hat{x} \in \hat{\Omega}$ we construct the path connecting $\hat{x}$ and $x_0$, which satisfies the conditions in Definition \ref{def-John}.

{\em Case 1}: $\hat{x} \in \Omega$. Noting that any Reifenberg flat domain is also a John domain, we take the same path as in Definition \ref{def-John}. Noting that for any $x\in\Omega$, $\dist(x,\hat{\Omega}^c)\geq \dist(x,\Omega^c)$, \eqref{eqn-John} is satisfied with the same $\lambda$.

{\em Case 2}: $\hat{x} \in \hat{\Omega}\setminus\Omega$. We assume that $\hat{x}$ lies in the extended cube $\hat{Q}_k$ with center $x_k$ and $\diam(\hat{Q}_k)=8r_k$. Let $x_0$ be the point defined in Definition \ref{def-John}. If $x_0\in Q_k$, we can take the straight line path. In this case, \eqref{eqn-John} is satisfied with the constant $7/(9\sqrt{d})$. Now, if $x_0\notin Q_k$, we first consider the straight line path
$$\gamma_1: [0,1/2]\mapsto \hat{Q}_k,\quad \gamma_1(0)=\hat{x},\quad \gamma_1(1/2)=x_k.$$
Since $\hat{x}\notin Q_k$, we have
$$\frac{1}{2}r_k\leq\abs{\gamma_1[0,1/2]}\leq 4\sqrt{d}r_k.$$
Noting that $x_k\in \Omega$, we consider the re-parametrized path coming from Definition \ref{def-John}:
$$\gamma_2:[0,1/2]\mapsto\Omega, \quad \gamma_2(0)=x_k,\quad \gamma_2(1/2)=x_0.$$
Take $\gamma =\gamma_1\circ\gamma_2$ be the path connecting $\gamma_1$ and $\gamma_2$. Now, when $t\in[0,1/2]$, again \eqref{eqn-John} is satisfied with the constant $1/\sqrt{d}$. When $t\in[1/2,1]$, we consider the following two cases: $\gamma(t)\in Q_k$ or $\gamma(t)\notin Q_k$.

If $\gamma(t)\in Q_k$, we have
\begin{equation*}
\begin{split}
\frac{\abs{\gamma[0,t]}}{\dist(\gamma(t),\hat{\Omega}^c)}
&\leq \frac{\abs{\gamma[0,1/2]}}{\dist(\gamma(t),\hat{\Omega}^c)} + \frac{\abs{\gamma[1/2,t]}}{\dist(\gamma(t),\Omega^c)}\\
&\leq \frac{4\sqrt{d}r_k}{7r_k/2} + \lambda^{-1}= \frac{8\sqrt{d}}{7}+\lambda^{-1}.
\end{split}
\end{equation*}
When $\gamma(t)\notin Q_k$, we have $\abs{\gamma[1/2,t]}\geq r_k/2$. Hence,
\begin{equation*}
\begin{split}
\frac{\abs{\gamma[0,t]}}{\dist(\gamma(t),\hat{\Omega}^c)}
&=\frac{\abs{\gamma[0,t]}}{\abs{\gamma[1/2,t]}}\cdot\frac{\abs{\gamma[1/2,t]}}{\dist(\gamma(t),\hat{\Omega}^c)}\\
&=(1+\frac{\abs{\gamma[0,1/2]}}{\abs{\gamma[1/2,t]}}) \cdot \frac{\abs{\gamma[1/2,t]}}{\dist(\gamma(t),\hat{\Omega}^c)}\\
&\leq (1+\frac{4\sqrt{d}r_k}{r_k/2})\cdot \frac{\abs{\gamma[1/2,t]}}{\dist(\gamma(t),\hat{\Omega}^c)}\\
&\leq (1+8\sqrt{d})\lambda^{-1}.
\end{split}
\end{equation*}
With all above, we have proved that $\hat{\Omega}$ is still a John domain. Now we extend $f$ to $\hat{\Omega}$ as
\begin{equation*}
\begin{cases}
\hat{f}:= f& \text{in }\Omega,\\
\hat{f}:= -\frac{1}{\abs{\hat{\Omega}\setminus\Omega}}\int_\Omega f & \text{in }\hat{\Omega}\setminus\Omega.
\end{cases}
\end{equation*}
Then we have
\begin{equation*}
\int_{\hat{\Omega}} \hat{f} = 0,\quad \norm{\hat{f}}_{L_{p_{*}}(\hat{\Omega})}\leq N(R_1,\abs{\Omega})\norm{f}_{L_{p_{*}}(\Omega)}.
\end{equation*}
Since $\hat{\Omega}$ is a John domain, we apply the result in \cite[Theorem 4.1]{ADM} to find $\phi=(\phi_1,\cdots,\phi_d)\in (W^{1,{p_{*}}}_0(\hat{\Omega}))^d$ satisfying
\begin{equation*}
D_i\phi_i = \hat{f}\quad \text{in }\hat{\Omega}, \quad \norm{\phi_i}_{W^{1,{p_{*}}}(\hat{\Omega})}\leq N(\diam(\hat{\Omega}),d,p)\norm{\hat{f}}_{L_{p_{*}}(\hat{\Omega})}.
\end{equation*}

Now by Sobolev inequalities and our construction of $\hat{\Omega},\hat{f}$, we obtain that $\phi\in (W^{1,p_{*}}_\cN(\Omega))^d$, and
\begin{equation*}
\norm{\phi_i}_{L_p(\Omega)}\leq N(\diam(\Omega),R_1,d,p)\norm{\hat{f}}_{L_{p_{*}}(\hat{\Omega})}.
\end{equation*}
The lemma is proved
\end{proof}

Now we are ready to give the proof of Theorem \ref{thm-bounded-domain}.
\begin{proof}[Proof of Theorem \ref{thm-bounded-domain}]
Using Lemma \ref{lem-solve-divergence-eqn}, for every $f\in L_{p_{*}}(\Omega)$, we can find $$(\phi_i)_{i=1}^d\in (W^{1,p_{*}}_\cN(\Omega))^d\subset (L_p(\Omega))^d$$ satisfying \eqref{eqn-est-div}. Now we consider the following problem:
\begin{equation}\label{eqn-no-div}
\begin{cases}
Lu= D_i (f_i+\phi_i) &\text{in }\, \Omega,\\
Bu =(f_i+\phi_i) \, n_i & \text{on }\, \cN,\\
u=0 & \text{on }\, \cD.
\end{cases}
\end{equation}
Since $\phi =0$ on $\cN$, one can easily check that any solution to \eqref{eqn-no-div} is also a solution to \eqref{180417@eq2}. Hence, without loss of generality, we may assume $f=0$.

We aim to use the Fredholm alternative. For this, we first introduce some operators. From Theorem \ref{thm-well-posedness}, for fixed large enough $\lambda$, we can find a unique weak solution $u\in W^{1,p}_\cD(\Omega)$ to \eqref{eqn-main-large-lambda} satisfying \eqref{eqn-general}, and hence \eqref{eqn-a-priori-no-dependence}.

We write $R(\lambda,L)$ as this solution operator, i.e.,
$$
R(\lambda,L):(L_p(\Omega))^d\times L_p(\Omega)\mapsto W^{1,p}_\cD(\Omega),\quad R(\lambda,L)(f_i,f)=u.
$$
In particular, for any $L_p$ function $f$, we write
$$R_\lambda(f):=R(\lambda,L)(0,f).$$
From \eqref{eqn-a-priori-no-dependence}, $R_\lambda$ is a bounded linear operator from $L_p(\Omega)$ to $W^{1,p}_\cD(\Omega)$. Denote $I$ as the compact embedding from $W^{1,p}_\cD(\Omega)$ to $L_p(\Omega)$. Now, we write
$$
T:W^{1,p}_\cD(\Omega)\rightarrow W^{1,p}_\cD(\Omega),\quad T(u)= R_\lambda\circ I(u).
$$
From our construction, $T$ is a compact operator. Noting that we have assumed that $f=0$, applying the operator $R(\lambda,L)$ to both sides of
$$(L-\lambda)u +\lambda u= D_i f_i,$$
we can rewrite \eqref{180417@eq2} as
\begin{equation}\label{eqn-fredholm}
(Id + \lambda T)u=R(\lambda,L)(f_i,0).
\end{equation}
By the Fredholm alternative, \eqref{eqn-fredholm} has a unique $W^{1,p}_\cD(\Omega)$ solution satisfying
$$
\norm{u}_{W^{1,p}(\Omega)}\leq N\norm{f_i}_{L_p(\Omega)},
$$
if the following homogeneous equation only has zero solution
\begin{equation}\label{eqn-zero}
\begin{cases}
Lv = 0& \text{in }\Omega,\\
Bv=0& \text{on }\cN,\\
v=0& \text{on }\cD.
\end{cases}
\end{equation}
When $v\in W^{1,2}(\Omega)$, this is true due to the weak maximum principle, noting that the proof in \cite[Section~8.1]{GT} actually shows that $\sup_{\overline{\Omega}}\abs{v}$ has to be achieved at the Dirichlet boundary. Hence the uniqueness of \eqref{eqn-zero} is proved for the case $p\geq 2$.

When $p<2$, we can use Theorem \ref{thm-well-posedness} and a bootstrap argument to improve the regularity. Suppose $v\in W^{1,p}(\Omega)$ is a solution to \eqref{eqn-zero}. Take $\lambda$ large enough, noting that $v$ is also a $W^{1,p}$ solution to
\begin{equation*}
\begin{cases}
(L-\lambda) v= -\lambda v &\text{in }\Omega,\\
Bv=0& \text{on }\cN,\\
v=0& \text{on }\cD.
\end{cases}
\end{equation*}
By the Sobolev embedding, $-\lambda v\in L_{pd/(d-p)}(\Omega)$. Take $p^{*}=\min\set{pd/(d-p),2}$. By the uniqueness of $W^{1,p^{*}}$ solutions in Theorem \ref{thm-well-posedness}, we obtain $v \in W^{1,p^{*}}$. Repeating this process if needed, in finite steps, we can reach $v\in W^{1,2}$. Hence we can use the weak maximum principle to deduce $v=0$ as before. This finishes the proof of the uniqueness. Hence, by the Fredholm alternative, Theorem \ref{thm-bounded-domain} is proved.
\end{proof}

\end{document}